\documentclass[a4paper,11pt,twoside]{article}	
\usepackage{amsmath,amssymb,amsthm,amsfonts,color,graphicx}
\usepackage{setspace}
\usepackage[margin=2.5cm]{geometry}
\usepackage{soul}
\newtheorem{theorem}{Theorem}
\DeclareMathOperator{\h}{\mathbb H^2}
\DeclareMathOperator{\hr}{\mathbb H^2\times\mathbb R}
\DeclareMathOperator{\hh}{\mathbb H^2}
\DeclareMathOperator{\D}{\mathbb D}
\DeclareMathOperator{\rr}{\mathbb R}
\DeclareMathOperator{\esf}{\mathbb S}
\DeclareMathOperator{\psl}{\widetilde{\rm PSL}_2 (\rr,\tau)}
\DeclareMathOperator{\ppsl}{\widetilde{\rm PSL}_2 (\rr,\tau)}
\DeclareMathOperator{\C}{\mathbb C}
\DeclareMathOperator{\sech}{sech}

\newtheorem{lemma}{Lemma}
\newtheorem{claim}{Claim}

\newtheorem{remark}{Remark}
\newtheorem{corollary}[theorem]{Corollary}
\theoremstyle{definition}\newtheorem{definition}{Definition}

\numberwithin{equation}{section}

\begin{document} 
\begin{title}
    {On the characterization of minimal surfaces with finite total
    curvature in $\hr$ and $\widetilde{\rm PSL}_2 (\rr)$}
\end{title}
\vskip .2in

\begin{author} {Laurent Hauswirth, Ana Menezes and Magdalena Rodr\'\i
    guez\thanks{Research partially supported by the MCyT-FEDER
      research project MTM2014-52368-P, MTM2017-89677-P, and by the GENIL research project
      no. PYR-2014-21 of CEI BioTic GRANADA.}}
\end{author}

\date{}

 \maketitle

\begin{abstract}
It is known that a complete immersed minimal surface with finite total
curvature in $\mathbb H^2\times\mathbb R$ is proper, has finite
topology and each one of its ends is asymptotic to a geodesic polygon
at infinity (Hauswirth and Rosenberg, 2006; Hauswirth, Nelli, Sa Earp
and Toubiana, 2015). In this paper we prove that these three
properties characterize complete immersed minimal surfaces with finite
total curvature in $\mathbb H^2\times\mathbb R$. As corollaries of
this theorem we obtain characterizations for minimal Scherk-type
graphs and horizontal catenoids in $\mathbb H^2\times\mathbb R$. We
also prove that if a properly immersed minimal surface in
$\widetilde{\rm PSL}_2(\mathbb{R},\tau)$ has finite topology and each
one of its ends is asymptotic to a geodesic polygon at infinity, then
it must have finite total curvature.
\end{abstract}

\section{Introduction}

The theory of finite total curvature minimal surfaces in $\hr$ has
been introduced by Collin and Rosenberg \cite{CR}. They remarked by
Fatou's convergence theorem in Gauss-Bonnet formula that complete
minimal graphs over ideal polygonal domain of $\h$ with a finite
number of vertices (called Scherk-type graphs) have finite total
curvature in $\hr$. Together with the vertical geodesic planes these
were the first examples appearing in the theory.  
Later, Hauswirth and Rosenberg~\cite{hr} proved that the total
curvature of these surfaces is a multiple of $2Ê\pi$. They also began
to describe their asymptotic geometric behavior at infinity and this
description has been later completed by Hauswirth, Nelli, Sa Earp and
Toubiana \cite{hnst}. They proved that any complete minimal surface
with finite total curvature in $\hr$ is proper, has finite topology
and each one of its ends is asymptotic to an admissible polygon at
infinity (see Definition~\ref{def:polygon} below).

The asymptotic boundary $\partial_{\infty}\mathbb H^2$ of $\mathbb
H^2$ can be identified with the unit circle. There are different
notions of asymptotic boundary $\partial _\infty (\hr)$ of
$\hr$. In this paper we use the product compactification obtained as
the product of the compactications of each one of the factors. This
is, we consider the following model for $\partial _\infty (\hr)$:
$((\partial_{\infty}\hh)\times[-1,+1])\cup(\hh\times\{\pm1\}),$ where
we represent the second factor $\mathbb R$ by some homeomorphism
$\phi: \mathbb R \to (-1,1)$.

We say that $p \in \partial _\infty (\hr)$ is in the asymptotic
boundary of a minimal surface ${\cal M}$ if there is a diverging
sequence of points $p_n \in {\cal M}$ such that $p_n$ converges to $p$
in the compactification. This means that if $p_n=(z_n, t_n) \in {\cal
  M}$ and $p=(a,h) \in \partial_\infty( \hr)$, then $z_n \to a$ in the
compactification of $\hh$ and $\phi (t_n) \to h$ in $[-1,1]$.

Given a vertical geodesic plane ${\cal M}= \alpha \times \mathbb R$,
where $\alpha$ is a horizontal geodesic with two endpoints $a_1,a_2
\in \partial _\infty \h$, we have $\partial _\infty {\cal M} =(\alpha
\times \{\pm1\})\cup (\{a_1,a_2\} \times [-1,1])$. This boundary can
be viewed as a quadrilateral curve at infinity. We generalize this
construction by the following definition.

\begin{definition}[Admissible polygon at infinity]
  We call {\em polygon at infinity} to any (connected, closed) polygon
  in $\partial_\infty(\hr)$ composed of a finite number of geodesics.
  We say that a polygon at infinity ${\cal P}$ is {\em admissible} if
  there exists an even number of geodesics $\alpha_1,\beta_1, ...,
  \alpha_k,\beta_k\subset\hh$ such that $\mathcal P$ is the union of
  the geodesics at infinity $\alpha_i\times\{+1\}$ and
  $\beta_i\times\{-1\}$, with $i=1, ..., k$, together with the
  corresponding vertical straight lines $L_i=\{a_i\}\times[-1,1]$,
  $a_i\in\partial_\infty\hh$, joining their endpoints (see
  Figure \ref{polygonal}).
  \label{def:polygon}
\end{definition}

\begin{definition}[Embedded Admissible polygon]
We say that an admissible polygon ${\cal P}$ is embedded if there
exists a one-to-one correspondence from $\esf^1$ to~${\cal P}$.
\end{definition}

We observe that the projection over $\hh$ of an embedded admissible
polygon at infinity can be non embedded, as Figure~\ref{Gamma1}-right
shows. This admissible polygon at infinity corresponds to the
asymptotic boundary of an example contructed by Pyo and the third
autor in~\cite{PR}, called a Twisted-Scherk example, that is a
properly embedded minimal disk with finite total curvature. In its
construction, we may turn $\alpha_2$ (see Figure~\ref{Gamma1}, right)
in the positive direction until it shares an endpoint with $\alpha_1$
and the other one with $\alpha_3$ ($\beta_3$ then shares an endpoint
with $\beta_1$ and the other with $\beta_2$). The polygon at infinity
we get is admissible but non embedded, and it corresponds to the
asymptotic boundary of a properly embedded minimal example.  This
example shows that the asymptotic boundary of a complete embedded
minimal surface with finite total curvature can be non embedded.

\begin{figure}[h]
 \centering
\includegraphics[height=4cm]{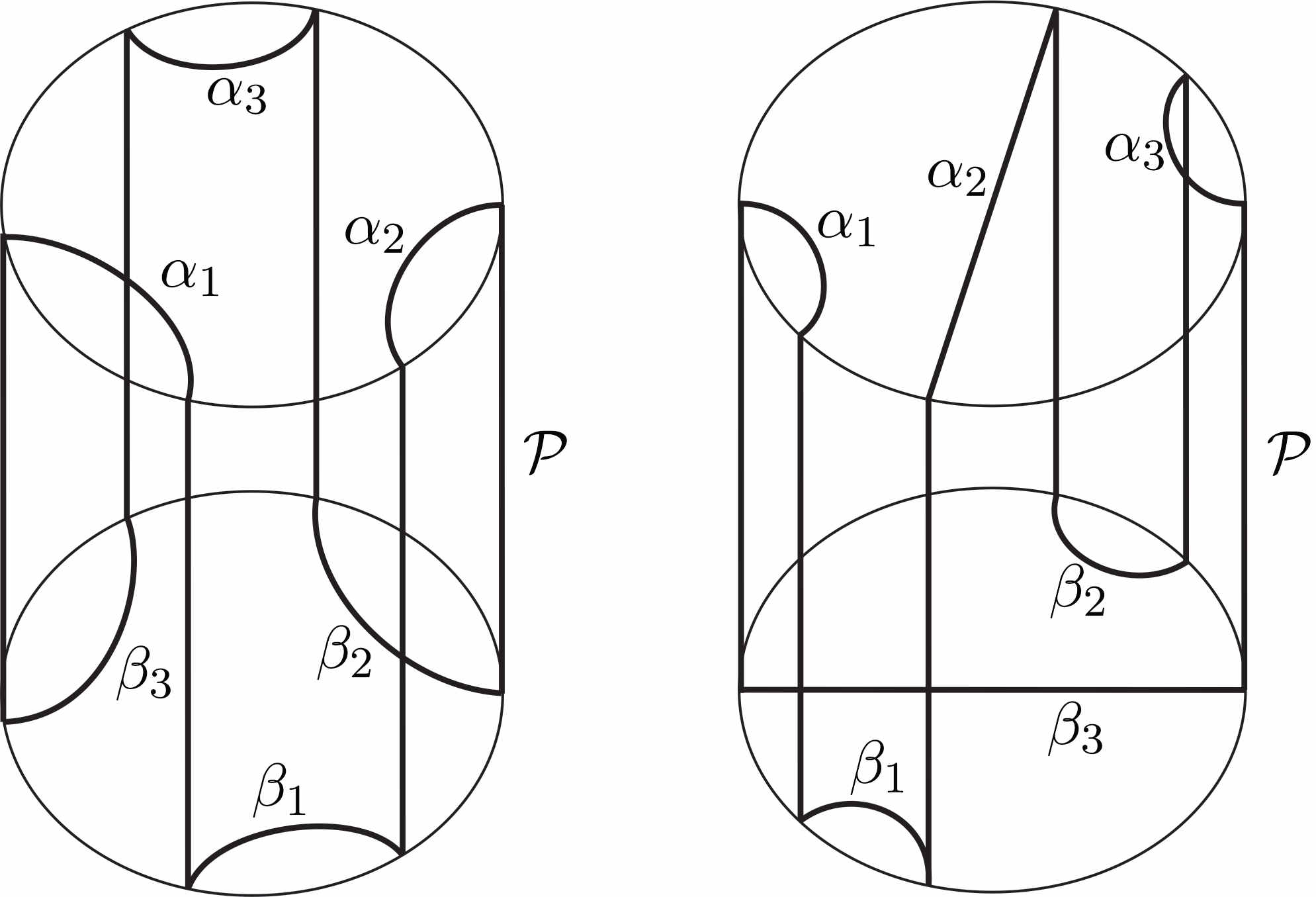}
\caption{Examples of embedded admissible polygons at infinity.}
\label{polygonal}
\end{figure}

Let us now describe the asymptotic behavior of finite total curvature
minimal surfaces in $\hr$.  A finite total curvature minimal surface
${\cal M}$ has finite topology, hence its ends are annular.  Let
$\Sigma$ be an annulus with the topology of $\esf^1\times(0,+\infty)$
and $X:\Sigma\rightarrow\hr$ be a proper minimal immersion. We call
$M=X(\Sigma)$. It is proved in~\cite[Lemma 2.3]{hnst} that, for $t_0
>0$ large enough, $M\cap\{t>t_0\})$ (resp. $M\cap\{t<-t_0\})$)
corresponds to a finite number of connected components $U_1, ...,U_k$
(resp. $V_1,...,V_k$) in $\Sigma$. Finite total curvature implies that
the curvature is uniformly bounded, converges uniformly to $0$ at
infinity and the tangent planes become vertical. For each $U_i$, there
exists a geodesic $\alpha_i \subset \h$ such that $X(U_i)$ is a
horizontal Killing graph (see Definition~\ref{killinggraph} below)
over $\alpha_i \times \mathbb R$ and $\partial_\infty X(U_i) \subset
\partial_\infty (\alpha_i \times \mathbb R)$ (similarly for $X(V_i)$,
for some geodesic $\beta_i\subset \h$). Moreover, for any vertical
line $\{a\} \times [-1,1]\subset {\cal P}$, there exists a horodisk
${\mathcal H}$ with $\partial_\infty {\mathcal H}=\{ a\}$ such that $M
\cap ({\mathcal H} \times \mathbb R)$ corresponds to a finite number
of connected components $W_1, ...,W_k$ in $\Sigma$. Each $X(W_i)$ is a
horizontal Killing graph over $\alpha \times \mathbb R$, where
$\alpha\subset \h$ is a geodesic having $a$ as an endpoint. Therefore,
this proves that there exists an admissible polygon at infinity
containing $\partial_\infty M$.

\begin{figure}[h]
 \centering
\includegraphics[height=3cm]{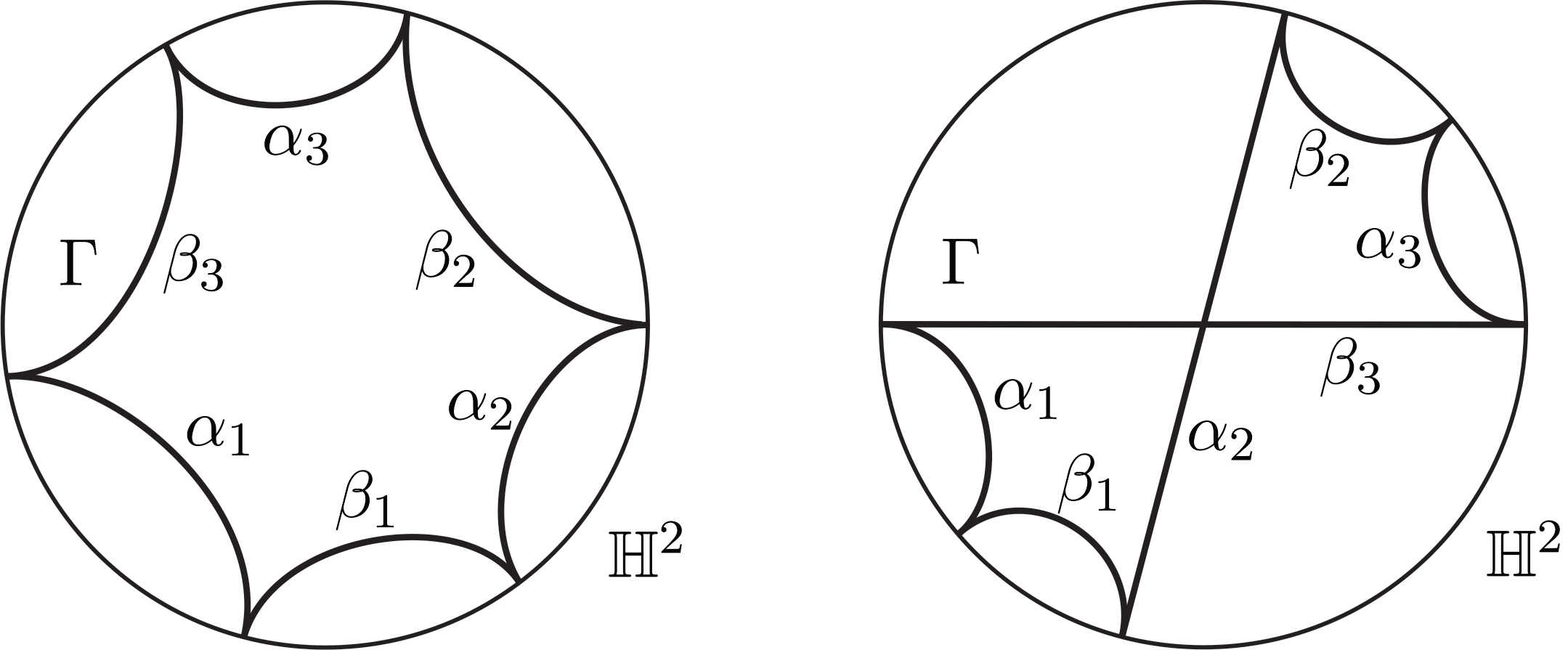}
\caption{Projection over $\hh$ of the embedded admissible polygons at
  infinity in Figure \ref{polygonal}.}
\label{Gamma1}
\end{figure}

The consequence of this behavior implies some classification theorems.
If the projection of ${\cal P}$ is embedded and the surface ${\cal M}$
is embedded and has finite total curvature, we can begin the
Alexandrov method of moving planes using horizontal slices coming from
above, and we obtain that the only one-end complete embedded minimal
surface ${\cal M}$ of finite total curvature with $\partial _\infty
{\cal M} = {\cal P}$ is a Jenkins-Serrin's type graph over the ideal
domain bounded by the projection of ${\cal P}$ (see
Theorem~\ref{th:graphs}). Another application is a Schoen's type
theorem for minimal annuli. Pyo \cite{p} and Morabito-Rodr\'\i guez
\cite{MoR} have constructed minimal annuli with total curvature
$4\pi$. The ends are asymptotic to two vertical geodesic planes
$\alpha_1 \times \rr$ and $\alpha_2 \times \rr$. These annuli are
called horizontal catenoids.

\begin{theorem}
\label{schoen}
\cite{hnst} A complete and connected minimal surface immersed in $\hr$
with nonzero finite total curvature and two ends, each one asymptotic
to a vertical geodesic plane, is a horizontal catenoid.
\end{theorem}
The subtle thing is to define correctly the notion of asymptotic
behavior at infinity of each end which permits to begin the Alexandrov
method of moving planes. This will imply that the annulus has three
geodesic planes of symmetry and is a geodesic bigraph (see
Definition~\ref{geodgraph} below) on each of them. This will be enough
to next conclude that the annulus is exactly a horizontal
catenoid. The asymptotic hypothesis and finite total curvature assumed
by the authors in Theorem \ref{schoen} can be rewritten in a strong
geometric hypothesis: Each end $M_i$ is assumed to be embedded and
additionally a horizontal Killing graph outside a compact set on some
vertical geodesic plane which converges uniformly to zero at
infinity. These hypotheses are similar to the one used in the original
work of R. Schoen for minimal surfaces in $\rr^3$ with two ends which
are graphs over non compact domains of some planes.  Alexandrov moving
planes technique can be initiate at infinity with this behavior of
$M_i$ (see \cite{hnst}). We remark that R. Sa Earp and E. Toubiana obtains
characterization of finite total curvature assuming stability, hence bounded
uniform curvature at infinity.

\medskip 

Let us now introduce some definitions that we use to define a weaker
notion of asymptoticity to an admissible polygon at infinity. We will
define this asymptoticity in a topological meaning rather than a
geometric one.  For that, we fix an admissible polygon at infinity
$\mathcal P$.  Given a point $a$ at infinity of $\hh$, we consider a
foliation given by a monotone family of horocylinders $\{{\mathcal
  H}(c)\} _{c\in\rr}$ with boundary $\{a \} \times[-1,1]$ at infinity.

\begin{definition}
  We say that $E$ is a {\it horizontal sheet} of $M\cap \bigcup_{c\geq c_0}{\mathcal H}(c)$, for some $c_0\in\rr$, if there exists a
  connected component $U$ of $X^{-1}(M\cap\{{{\mathcal H}(c)}; c\geq
  c_0\})$ such that $E=X(U)$.
  \label{def:asymp}
\end{definition}

\begin{definition}
  We say that $E$ is a {\it  vertical  sheet} of $M\cap\{t>t_0\}$
  (resp. $M\cap\{t<-t_0\}$), for some $t_0\in\rr$, if there
  exists a connected component $U$ of $X^{-1}(M\cap\{t>t_0\})$
  (resp. $X^{-1}(M\cap\{t<-t_0\})$) such that $E=X(U)$.
  \label{def:asymp}
\end{definition}

We observe that a horizontal (resp. vertical) sheet $E$ is not
necessarily a connected component of $M\cap\{{{\mathcal H}(c)}; c\geq
c_0\}$ (resp. $M\cap\{t>t_0\}$), since we do not assume $M$
necessarily embedded.

\begin{definition}
We say that $M$ is {\it asymptotic} to an embedded admissible polygon at
 infinity ${\cal P}$ if $\partial_\infty M\subset{\cal P}$.
 
 If ${\cal P}$ is not embedded, we say that $M$ is {\it asymptotic} to
 the admissible polygon at infinity ${\cal P}$ if $\partial_\infty
 M\subset{\cal P}$ and there exists $t_0>0$ such that the following
 assertions hold:
    \begin{itemize}
    \item For any vertical sheet $E$ of $M\cap\{t>t_0\}$, there exists some
      $i=1, ..., k$ such that $\partial_\infty
      E\subset\partial_\infty(\alpha_i\times\rr)$.
    \item For any vertical sheet $E$ of $M\cap\{t<-t_0\}$, there exists some
      $i=1, ..., k$ such that $\partial_\infty
      E\subset\partial_\infty(\beta_i\times\rr)$.
    \end{itemize}
  \label{def:asymp}
\end{definition}

\begin{remark} There is no assumption on horizontal sheets in the last definition.
\end{remark}

This notion of asymptoticity is topological. Next we define a stronger definition of
asymptoticity which shares more geometric hypothesis. 

\begin{definition}
We say that ${ M}$ is an {\it asymptotic multigraph at infinity} to
${\mathcal P}$ if $M$ is {\it asymptotic} to ${\mathcal P}$ and any
vertical and horizontal sheet $E$ can be written outside a compact set
as a horizontal graph over a domain of $\alpha \times \rr$, for some
well chosen geodesic $\alpha$. The {\it asymptotic multigraph} is {\it
  Killing} (see definition 7) if $E$ is a graph along horocycles
orthogonal to $\alpha \times \rr$ while the {\it asymptotic
  multigraph} is {\it geodesic} (see definition 8) if $E$ is a graph
along geodesic orthogonal to $\alpha \times \rr$.
\end{definition}

R. Sa Earp and E. Toubiana in \cite{ET1,ET2} has studied 
characterization of finite total curvature assuming stability.
Stability say that the curvature on ends is bounded and using appropriate 
barriers, they can consider surfaces asymptotically multigraph at infinity
as in the previous definition or graph on some slice $\hh \times \{0\}$.

In the following theorem we prove that the weak condition of asymptoticity
is sufficient to prove finite total curvature. We provide uniform bound of the
curvature by proving that properly immersed surfaces are asymptotically multigraph 
at infinity converging uniformly to zero.

\begin{theorem}
  Let ${\cal M}\subset \hr$ be a properly immersed minimal surface
  with finite topology and possibly compact boundary. Suppose that
  each end $M$ of ${\cal M}$ is asymptotic to an admissible polygon at
  infinity. Then $M$ is both a geodesic and Killing asymptotic
  multigraph at infinity and ${\cal M}$ has finite total curvature.
  \label{thm-2}
\end{theorem} 

Theorem \ref{thm-2} is an immediate consequence of the next theorem
using the fact that finite topology implies that each end $M$ is
annular.

\begin{theorem}\label{th:mainhr}
  Let $M\subset \hr$ be a properly immersed minimal annulus with one
  compact boundary component $\partial M$ and asymptotic to an
  admissible polygon ${\cal P}$ at infinity. Then $M$ is an asymptotic
  geodesic and Killing multigraph at infinity and has finite total
  curvature.
\end{theorem}

As a consequence we have a characterization for finite total curvature
minimal surfaces:

\begin{theorem}
  A complete minimal surface of $\hr$ has finite total curvature if
  and only if it is proper, has finite topology and each one of its
  ends is asymptotic to an admissible polygon at infinity.
\end{theorem} 

We also obtain the following uniqueness result derived from
Theorems~\ref{schoen} and~\ref{th:mainhr}:

\begin{theorem}
A complete (and connected) minimal surface properly immersed in $\hr$
with two embedded ends $E_1$ and $E_2$ satisfying $\partial _\infty
E_1 \subset \partial_\infty (\alpha_1 \times \rr)$ and $\partial
_\infty E_2 \subset \partial_\infty (\alpha_2 \times \rr),$ for some
geodesics $\alpha_1$ and $\alpha_2$ in $\h$, must be a horizontal catenoid.
\end{theorem}

We can also prove using Theorem~\ref{th:mainhr} and Alexandrov's
moving planes method the following results:

\begin{theorem}\label{th:symmetry}
Let $M$ be a (connected) properly immersed minimal surface in $\hr$
with a finite number of embedded ends $E_1,\dots, E_k$ satisfying
$\partial _\infty E_i \subset \partial_\infty (\alpha_i \times \rr)$
for any $i=1,\dots,k$, where $\alpha_1,\dots, \alpha_k$ denote
complete geodesics in $\h$ cyclically ordered. Then $M$ is a vertical
bigraph symmetric with respect to a horizontal slice.
\end{theorem}

\begin{theorem}\label{th:graphs}
Let $M$ be a properly embedded minimal surface in $\hr$ with finite
topology and one end asymptotic to an admissible polygon at infinity
${\cal P}$. Suppose that the vertical projection of ${\cal P}$ in
$\hh$ is the boundary of a convex domain $\Omega$. Then $M$ is a
vertical graph.

In particular, if $\alpha_i\times\{1\}$ and
$\beta_i\times\{-1\}$, with $i=1,\dots, k$, are the edges of
${\cal P}$ then:
\begin{enumerate}
\item $\sum_{i=1}^k |\alpha_i|= \sum_{i=1}^k |\beta_i|$; and
 \item for any inscribed polygonal domain $D$ in $\Omega$,
$\sum_{i=1}^k |\alpha_i\cap\partial D|= \sum_{i=1}^k
 |\beta_i\cap\partial D|$,
\end{enumerate}
 where $|\bullet|$ denotes the hyperbolic length of the curve
 $\bullet$.
\end{theorem}

\noindent
{\bf Open problems}.  Does it exist a one-end torus of finite total
curvature? Does it exist one whose associated polygon at infinity
${\cal P}$ does not project on some embedded ideal polygonal (see
Figure \ref{polygonal}, right)? What about higher genus? More
generally, can we study the moduli space of finite total curvature
minimal surfaces in function of the space of admissible polygons at
infinity?

\vskip 0.5cm Theorem \ref{thm-2} has a natural extension to
$\psl$. These simply-connected homogeneous manifolds can be viewed as
$\D\left(\sqrt{\frac{-4}{\kappa}}\right)\times \rr$, where
$\D\left(\sqrt{\frac{-4}{\kappa}}\right)= \{x^2+y^2 \leq -4/
\kappa\}$, endowed with the following metric:
\[
g= \lambda ^2 (dx^2+dy^2)+(\tau \lambda (ydx -xdy)+dz)^2,
\]
 where $\lambda = \frac{1}{1 + \frac{\kappa}{4} (x^2+y^2)}$. R. Younes
 \cite{younes} first studied the Jenkins-Serrin problem on compact
 domains of the basis and S. Melo \cite{melo} proved the existence of
 complete minimal graphs on ideal domains.
 
 The curvature of a minimal surface in $\psl$ satisfies $K \leq \tau$
 and does not have necessarily a negative sign. Hence it seems more
 difficult to prove theorems involving the Gaussian curvature. However in \cite{mn}
 Minh Nguyen proves that minimal surfaces with uniform bounded
 curvature which are geodesic asymptotic multigraphs have finite total
 curvature
 \[
 \int_\Sigma |K| dA \leq C.
 \]
 Using this property we can prove Theorem \ref{th:mainhr} in $\psl$:

\begin{theorem}\label{th:mainpsl}
  Let $M \subset \psl$ be a properly immersed minimal annulus with one
  compact boundary component $\partial M$ and asymptotic to an
  admissible polygon ${\cal P}$ at infinity. Then $M$ is an asymptotic
  Killing and geodesic multigraph and has finite total curvature.
\end{theorem} 

This theorem implies that the complete graphs defined over ideal
polygonal domains of $\{ z=0\}$ constructed by Melo \cite{melo} have
finite total curvature. Collin, Nguyen and the first author have
constructed in \cite{MCH}, via variational methods, a horizontal
catenoid in $\psl$ asymptotic to two vertical geodesic planes. As a
consequence of Theorem \ref{th:mainpsl}, this annulus has finite total
curvature.

\begin{remark} 
It is not known if a complete finite total curvature annular end in
$\psl$ must be asymptotic to an admissible polygon at infinity.
\end{remark}

\section{Preliminaries}
\label{sec:prel}

There are several models for the 2-dimensional hyperbolic space
$\hh$. If we use the Poincar\'e disk model for the 2-dimensional
hyperbolic space, then the space $\hh$ is given by
\[
\mathbb H^2=\{(x,y)\in\rr^2;\ x^2+y^2<1\}
\]
with the hyperbolic metric $g_{-1}=\frac{4}{(1-x^2-y^2)^2}g_0,$ where
$g_0$ denotes the Euclidean metric in $\rr^2.$ In this model, the
asymptotic boundary $\partial_{\infty}\mathbb H^2$ of $\mathbb H^2$
can be identified with the unit circle. There are different models to
describe the asymptotic boundary of $\hr$. In this paper we consider
the product compactification obtained as the product of the
compactifications of each of the factors.

\medskip If we consider the half-plane model for $\hh$, then the space
$\hh$ is given by
$$\hh=\{(x,y)\in\rr^2;y>0\},$$
endowed with the metric $g_{-1}=\frac{1}{y^2}g_0$ 

To describe the homogeneous $3$-manifold $\psl$ we use the half-plane
model for $\hh,$ since we are interested in horizontal graphs. Hence,
in Euclidean coordinates, we have
\[\psl=\rr^2_+\times\rr=\{(x,y,t)\in\rr^3;y>0\},\]
endowed with the Riemannian metric
\[ds^2=\frac{1}{y^2}g_0+\left( d t-\frac{2\tau}{y}d x\right)^2.\]
The orthonormal frame $B=\{E_1,E_2,E_3\}$ in $\psl$ is given by 
\begin{align*}
E_1&=y\partial_x+2\tau \partial_t,&
E_2&=y\partial_y&
E_3&=\partial_t
\end{align*}
and satisfies $[E_1,E_3]=[E_2,E_3]=0$, $[E_1,E_2]=-E_1+2\tau E_3$, so the Levi-Civita connection is given by
\begin{equation}\label{eqn:killing-conexion}
\begin{array}{lll}
\overline\nabla_{E_1}E_1=E_2,&\overline\nabla_{E_1}E_2=-E_1+\tau E_3,&\overline\nabla_{E_1}E_3=-\tau E_2,\\
\overline\nabla_{E_2}E_1=-\tau E_3,&\overline\nabla_{E_2}E_2=0&\overline\nabla_{E_2}E_3=\tau E_1,\\
\overline\nabla_{E_3}E_1=-\tau E_2,&\overline\nabla_{E_3}E_2=\tau E_1,&\overline\nabla_{E_3}E_3=0.
\end{array}
\end{equation}

From now on, $N^3$ will denote $\hr$ or $\psl$. 

 We consider a vertical geodesic plane $\alpha \times \rr=\{ (x,y,t)
 \in \rr^2_+\times\rr; x=0\}$ and we define different notions of
 horizontal graphs over $\alpha\times\rr$:

\begin{definition}
\label{killinggraph}
A surface ${\cal M}\subset N^3$ is said to be a horizontal Killing
graph over $\alpha \times \rr$, if ${\cal M}$ is the graph of a
function $f: \Omega \subset \alpha \times \rr \to \rr$ along
horizontal horocycles, that is, ${\cal M}=X(\Omega),$ where
$$X(y,t)= (f(y,t), y, t).$$
\end{definition}

The parabolic isometry preserving the point at infinity $(0,0,0)$
induces a Killing field into the ambient space and a positive Jacobi
field on the graph ${\cal M}$. A well known result by Fischer-Colbrie
and Schoen \cite{Fischercolbrie} assures that the existence of a
positive Killing field on the surface ${\cal M}$ gives that ${\cal M}$
is stable, hence has bounded curvature away from its boundary by an uniform
estimate (R. Schoen \cite{S}).

\begin{definition}
\label{geodgraph}
A surface ${\cal M}\subset N^3$ is a horizontal geodesic graph over $\alpha
\times \rr$, if ${\cal M}$ is the graph of a function $f: \Omega
\subset \alpha \times \rr \to \rr$ along horizontal geodesics
orthogonal to $\alpha \times \rr,$ that is, ${\cal M}=X(\Omega),$
where
$$X(y,t)= (e^y \tanh (f), e^y \sech (f),t + 2 \tau \arcsin (\tanh (f))).$$
\end{definition}

The mean curvature operator associated to this notion of graphs has
been studied in \cite{mn}.  In this case, we cannot use stability to
assure that a geodesic graph has bounded curvature away from its
boundary but we can apply a blow up argument inspired by Rosenberg,
Souam and Toubiana \cite{RST}:

\begin{lemma} 
\label{bcurv}A horizontal geodesic graph ${\cal M}$ over $\alpha \times \rr$ has uniform
bounded curvature.
\end{lemma}

\begin{proof}
This blow up argument is standard.  Suppose ${\cal M}$ does not have
bounded curvature. Then there exists a divergent sequence $\{p_n\}$ in
${\cal M}$ such that $|A(p_n)|\geq n,$ where $A$ denotes the second
fundamental form of ${\cal M}$.  Denote by $C_n$ the connected
component of $p_n$ in an extrinsec ball $\bar B(p_n,\delta)\cap {\cal
  M}$, for some $\delta>0$. Consider the function
$f_n:C_n\rightarrow\rr$ given by
$$
f_n(q)=d(q,\partial C_n)|A(q)|,
$$ where $d$ is the extrinsic distance. The function $f_n$ restricted
to the boundary of $C_n$ is identically zero and
$f_n(p_n)=\delta|A(p_n)|>0.$ Then $f_n$ attains a maximum in a point
$q_n$ of the interior. Hence $\delta|A(q_n)|\geq d(q_n,\partial
C_n)|A(q_n)|=f_n(q_n)\geq f_n(p_n)=\delta|A(p_n)|\geq \delta n,$ what
yields $|A(q_n)|\geq n.$

Now consider $r_n=\frac{d(q_n,\partial C_n)}{2}$ and denote by $B_n$
the connected component of $q_n$ in $\bar B(q_n,r_n)\cap {\cal M}.$ We
have $B_n\subset C_n.$ If $q\in B_n,$ then $f_n(q)\leq f_n(q_n)$ and
$$
\begin{array}{rcl}
d(q_n, \partial C_n)&\leq & d(q_n,q) + d(q, \partial C_n)\\
&&\\
&\leq & \frac{d(q_n,\partial C_n)}{2} +d(q, \partial C_n)\\
&&\\
\Rightarrow d(q_n, \partial C_n) &\leq & 2d(q,\partial C_n).

\end{array}
$$ 
Hence we conclude that $|A(q)|\leq 2|A(q_n)|.$

Call $g$ the metric on $N^3$ in the half space model. The graph ${\cal
  M}$ is transverse to a foliation of horocylinders (vertical planes
in the euclidean model).  Consider ${\cal M}_n$ the homothety of $B_n$
by $\lambda_n=|A(q_n)|$. We obtain at the limit a complete minimal
surface $\widetilde{\cal M}$ in $\rr^3$ which is transversal to the
limit of the foliation of circle dilated by the homothety. This
foliation is converging to parallel line in $\rr^3$, hence
$\widetilde{\cal M}$ is a complete graph of $\rr^3$, which is flat by
Bernstein theorem, a contradiction.

\end{proof}

We now introduce some techniques we will use in this paper. Roughly
speaking we will use the hypothesis on the asymptotic boundary of the
ends to construct barriers which will constrain the geometry of the
ends locally in some union of vertical slabs (i.e. the region bounded
by the equidistant planes at a fixed distance from a vertical geodesic
plane). After that we will study the geometry of subdomains of ends
which are contained in some vertical slab with small width. We will
use what we call Dragging Lemma, a technique developped by Collin,
Hauswirth and Rosenberg in \cite{CollinHR2,CollinHR}, to prove that
the end is a horizontal multigraph, hence has uniform bounded
curvature at infinity using stability or Lemma \ref{bcurv}.  Then we
will prove that the property on the boundary at infinity implies that
the ends have finite total curvature.

\subsection{Family of barriers: ${\cal C} (h)$ and $S_h$}
\label{subsub:Sh}

Given $h\in(0,\pi)$, there is a one parameter family of rotationally
invariant surfaces ${\mathcal C}(h)$ called vertical catenoids in $\hh
Ê\times \rr$, where $h$ is the total height of the examples. The
boundary of ${\mathcal C}(h)$ at infinity consists of two horizontal
circles at height $t =h/2$ and at $t=-h/2$. We denote by $r_h$ the
size of the neck (the length of the curve ${\cal C} (h)\cap \{t=0\}$)
of these examples. When $h \to \pi$, then~$r_h \to \infty$ and the
surface disappears at infinity.  When $h \to 0$, $r_h\to 0$ and the
catenoids converge to the horizontal section $\hh \times \{ 0 \}$ with
degree two.

In $\psl$, Pe\~nafiel \cite{pe} has studied rotationally invariant
families of minimal surfaces. There are vertical catenoids whose
boundary at infinity consists of two horizontal circles at heights
$\pm h$, for any $h\in(0, \pi\sqrt{1+4\tau^2})$, with $r_h \to \infty$
as $h \to \pi \sqrt{1+4\tau^2}$.

Applying a maximum principle with these families of rotationally
invariant surfaces we have the following non existence property:

\begin{lemma}
\label{rotcatenoid}
There is no minimal surface ${\cal M}$ in $N^3$ with compact
instersection with 
$$S= \{ -\pi \sqrt{1+4\tau^2}/ 2\leq t \leq \pi
\sqrt{1+4\tau^2}/ 2\}$$ and boundary $\partial {\cal M} \cap S
=\emptyset$.
\end{lemma}

Given $h>\pi$ and a geodesic $\alpha\subset\hh$ with endpoints
$a_1,a_2\in\partial_\infty\hh$, we will consider the minimal surface
$S_h\subset\hh\times(0,h)$, first introduced by Hauswirth \cite{h}
then by Sa Earp and Toubiana \cite{ST} and Daniel \cite{d} (see also
Mazet, Rodr\'iguez and Rosenberg \cite{mrr,MazetRodriguezRosenberg}).
This minimal surface is a vertical bigraph with respect to
$\hh\times\{h/2\},$ is invariant by horizontal translations along
$\alpha$ and its asymptotic boundary is $(\eta\times\{0,h\})\cup
(\{a_1,a_2\}\times[0,h])$, where $\eta$ is an arc in
$\partial_\infty\hh$ with endpoints $a_1,a_2$ (see Figure
\ref{figure1} (left)). We remark that for each $t\in(0,h)$, $S_h\cap
(\hh\times\{t\})$ is an equidistant curve to
$\alpha\times\{t\}$. Moreover, when $h$ goes to $+\infty$, the
distance between $S_h\cap (\hh\times\{h/2\})$ and
$\alpha\times\{h/2\}$ goes to zero; in fact, the graph
$S_h\cap\{0<t<h/2\}$ converges to the minimal graph $S_\infty$ defined
over the domain bounded by $\alpha\cup\eta$ with boundary values
$+\infty$ over $\alpha$ and 0 over $\eta$ (see Figure \ref{figure1}
(right)).

\begin{figure}[h]
  \centering
  \includegraphics[height=3cm]{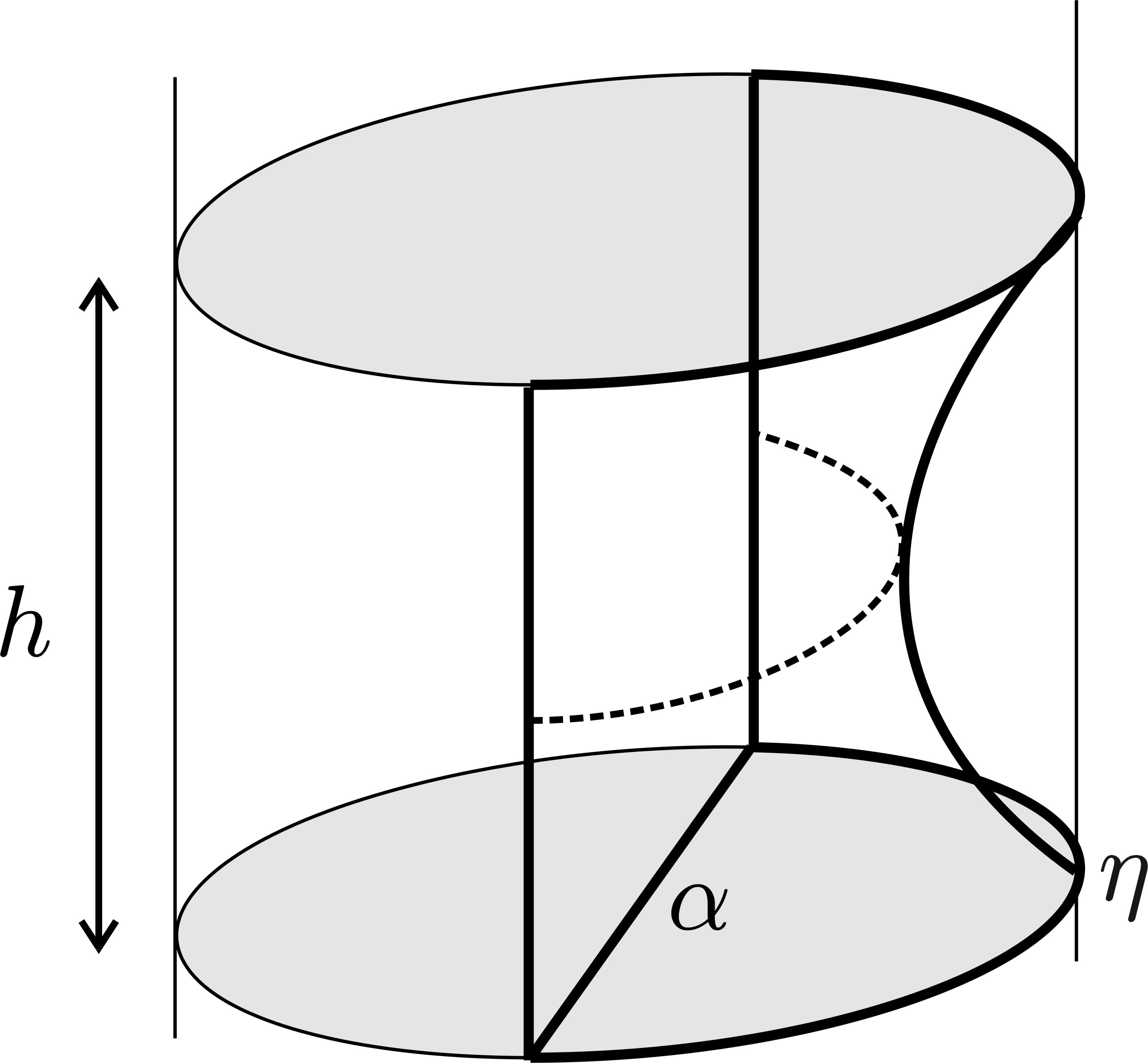}
  \hspace{2cm}
   \includegraphics[height=4cm]{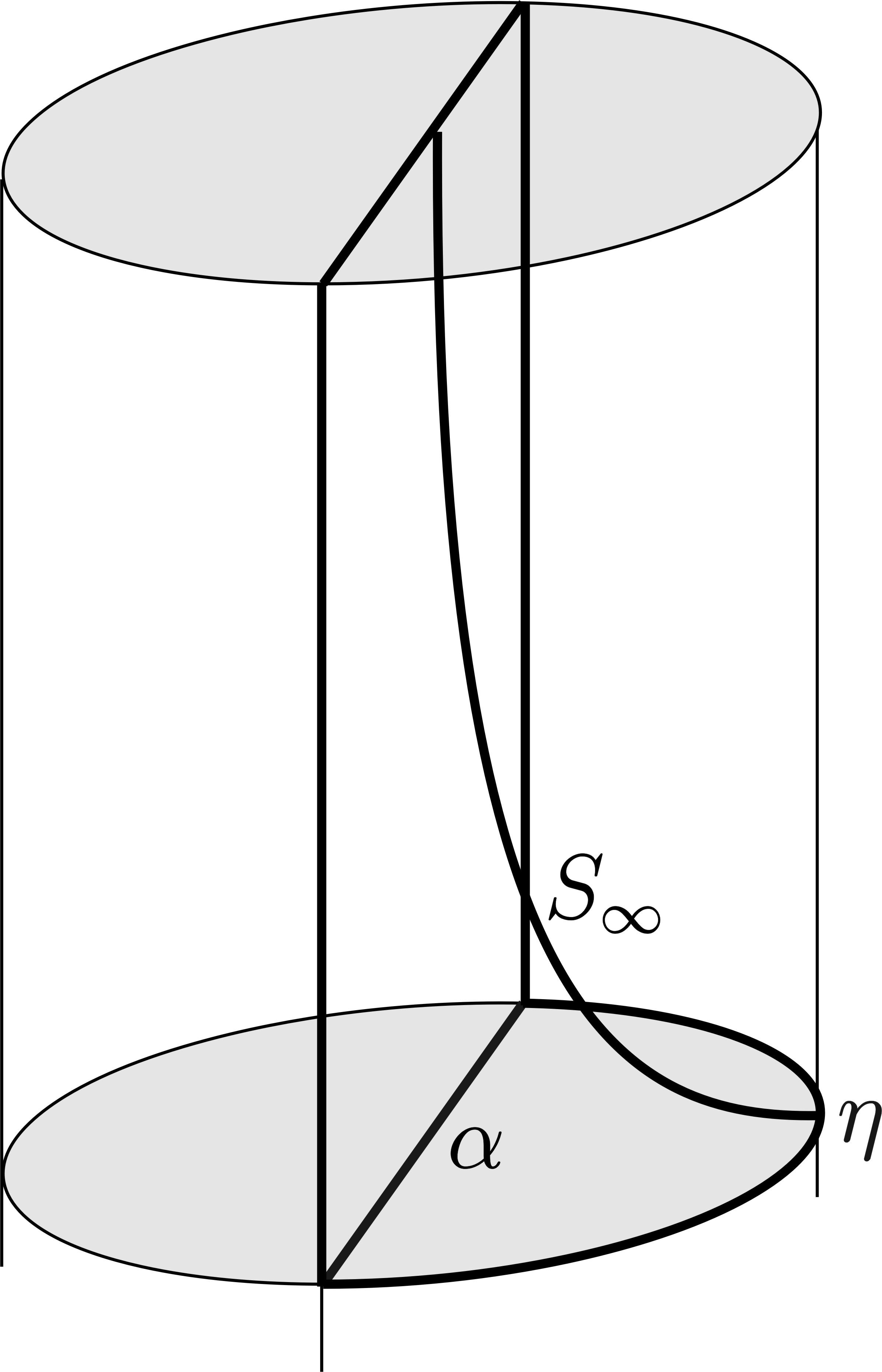} 
  \caption{Left: Minimal surface $S_h$; Right: Minimal surface $S_\infty$}
    \label{figure1}
\end{figure}

In $\psl$, Folha and Pe\~nafiel~\cite{FP} have obtained similar minimal
disks~$S_h$, for any $h>\pi\sqrt{1+4\tau^2}$, and $S_\infty$.

\subsection{The Dragging Lemma}

This refer as a technique of topological continuation of an arc into a
minimal surface immersed in an arbitrary three manifold $N^3$ (in our
context, $N^3$ is $\hh \times \rr$ or $\psl$).  We deal with a
geometrical situation where the immersed minimal surface ${\cal
  M}=X(\Sigma)$ is simply connected and is contained in a slab $S$
with small width.  In this situation, we consider a compact annulus
$A_0$ with boundary outside the slab, intersecting ${\cal M}$ in its
interior with $A_0 \cap \partial {\cal M} = \emptyset$. The two
components of the boundary of $A_0$ are contained in different
connected components of $N^3 \setminus S$.

Now we move the annulus $A_0$ by using an ambient isometry and we
consider $t \to A(t)$ the position of the translated annulus moved in
a ${\mathcal C}^1$-way such that $A(t) \cap \partial {\cal M} =
\emptyset $ and $\partial A(t) \cap {\cal M}= \emptyset$. The maximum
principle says that $A(t) \cap {\cal M}$ cannot be empty, otherwise
there would be a last point of contact between these two surfaces, a
contradiction. Then if $p$ is a point in the intersection at $t=0$, we
can construct a path $\alpha (t)$ with endpoint $p$ such that $\alpha
(t) \in A(t) \cap {\cal M}$. This path can be continued while $A(t)$
does not meet the boundary of ${\cal M}$. This path can be constructed
in a ${\mathcal C}^1$-way and monotonically. This is a consequence of
the Dragging Lemma:

\begin{lemma}[Dragging Lemma \cite{CollinHR2,CollinHR}] Let $g:\Sigma
  \to N^3$ be a properly immersed minimal surface in a complete
  $3$-manifold $N^3$. Let $A$ be a compact surface (perhaps with
  boundary) and $f: A \times [0,1] \to N^3$ be a ${\mathcal C}^1$-map
  such that $f(A\times\{ t \})=A(t)$ is a minimal immersion for $0
  \leq t \leq 1$. If $\partial(A(t)) \cap g(\Sigma) = \emptyset$ for
  $0\leq t \leq 1$ and $A(0) \cap g(\Sigma) \neq \emptyset $, then
  there is a ${\mathcal C}^1$ path $\alpha (t)$ in $ \Sigma$ such
  that $(g \circ \alpha)(t) \in A(t) \cap g( \Sigma) $ for $0 \leq t
  \leq 1$. Moreover, we can prescribe any initial value $(g \circ \alpha)
  (0) \in A(0) \cap g(\Sigma)$.
\label{DL}
\end{lemma}

\subsection{The minimal annulus $A_0$ and an application of the Dragging Lemma}
\label{sec-annulus}
We consider a slab $\mathcal{R}_d$ in $\hh \times \rr$ bounded by some
equidistant planes $P^d$ and $P^ {-d}$ to some vertical plane). We will study the geometry of subdomains of ends which are
contained in some vertical slab with small width. We will use what we
call Dragging Lemma, a technique developped by Collin, Hauswirth and
Rosenberg in \cite{CollinHR2,CollinHR}, to prove that the end is a
horizontal multigraph, hence has uniform bounded curvature at infinity
using stability or Lemma \ref{bcurv}.  Then we will prove that the
property on the boundary at infinity implies that the ends have finite
total curvature.

\begin{lemma} Given $d>0$ small enough, there exists a compact stable
  minimal annulus $A_0$ in $\hhÊ\times \rr$ bounded by two large
  enough circles (in exponential coordinates) $\eta_+\subset P^d$ and
  $\eta_{-}\subset P^{-d}$. This annulus $A_0$ is symmetric with
  respect to the vertical geodesic plane $\alpha_{0}\times\rr$ and the
  unit normal vector to $A_0$ along the intersection curve $A_0\cap
  (\alpha_{0}\times\rr)$, which is convex, takes all directions in the
  plane $\alpha_{0}\times\rr$.
\end{lemma}

 Let us denote by $\xi$ the geodesic that joins the centers of
 $\eta_+$ and $\eta_{-},$ and by $q_1$ the point $\xi\cap
 (\alpha_0\times\rr).$ Assume the circles $\eta_+$ and $\eta_{-}$ are
 sufficiently close so that $\xi\cap A_0=\emptyset.$

\begin{figure}[h]
    \centering
    \includegraphics[height=9cm]{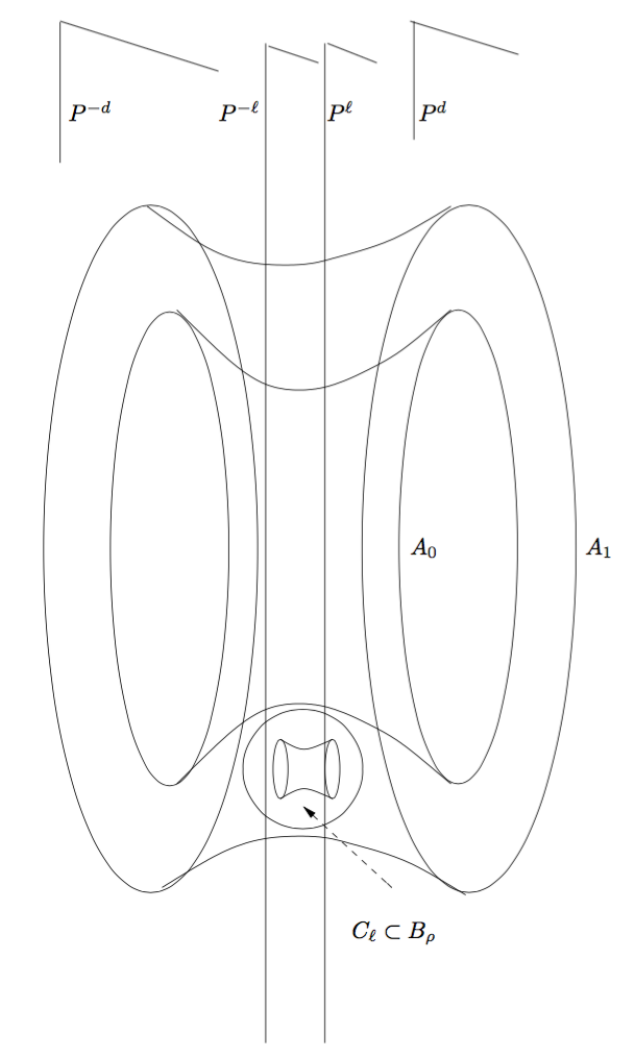}
    \caption{Annuli $A_0$ and $C_\ell$}
    \label{fig:annuli}
  \end{figure}

By stability, there exist $\delta>0$ and a foliation of a (closed)
neighborhood of $A_0$ in the slab $\mathcal {R}_d$ bounded by the
vertical planes $P^d \cup P^{-d}$ given by compact
annuli $A_s$, for $-\delta\leq s\leq \delta$, each $A_s$ with boundary
$\eta_{s}\cup\eta_{-s}$, where $\eta_s$ and $\eta_{-s}$ are
equidistant curves at distance $s$ from $\eta_+$ and $\eta_{-},$
respectively.  $A_0$ separates the slab $\mathcal {R}_d$ in two
connected components, one interior compact region $A_0^-$ and the
other one $A_0 ^+,$ the outside region of $ \mathcal {R}_d\setminus
A_0$ which is non compact.  Let us assume that $A_s\subset A_0^+$ if
$s>0$.  Let us write
\[
{\rm Tub}^+(A_0)=\cup_{s\in [0,\delta]}A_s\quad\mbox{and}\quad
{\rm Tub}^-(A_0)=\cup_{s\in [-\delta,0]}A_s.
\]
There exists a small constant $\rho>0$ such that, for
any point $q\in A_{\delta/2}\cap(\alpha_{0}\times\rr)$
(resp. $A_{-\delta/2}\cap(\alpha_{0}\times\rr)$), the geodesic open
ball $B_{\rho}(q)$ centered at $q$ with radius $\rho$ is contained in
Tub$^+(A_0)$ (resp. Tub$^-(A_0)$) and any such $B_{\rho}(q)$ contains
a small compact minimal annulus $C_\ell$ bounded by two circles (in
exponential coordinates) contained in $B_{\rho}(q)\cap
(P^\ell \times\rr)$ and $B_{\rho}(q)\cap (P^{-\ell}\times\rr)$, for
some small $\ell>0$ (see Figure \ref{fig:annuli}).  We further can take
$\rho>0$ satisfying:
\begin{enumerate}
\item $3\rho< \mbox{dist}(A_0,\xi),$
\item $B_{2\rho}(q_1)\cap \mbox{Tub}^{-}(A_0)=\emptyset,$ where
  $q_1=\xi\cap (\alpha_0\times\rr)$.
\end{enumerate}

\begin{lemma}\label{lema}
If there is a compact minimal surface ${\cal M} \subset \mathcal
{R}_\ell \cap A_0^+$ with $\partial {\cal M} \subset A_0$, then ${\cal
  M}$ is actually a subdomain of $A_0$.
\end{lemma}

\begin{proof}
Let $K$ be the compact set containing $A_0$ and bounded by horizontal
and vertical planes. The surface ${\cal M}$ cannot have points outside
$K$ without having an interior point of contact with a horizontal
slice or a vertical geodesic plane, contradicting the maximum
principle.  Moreover, ${\cal M}$ cannot be entirely contained in ${\rm
  Tub}^+(A_0)$, otherwise there would be a last leaf of the foliation
having a last point of contact, a contradiction again with the maximum
principle.

However, it still remains some room between the last annulus of the
foliation $A_s$ and $\partial K$. But we can find a catenoid $C_\ell$
which intersects ${\cal M}$ without intersecting the boundary by
choosing correctly the point $q$ on the waist circle of $A_{\delta/2}$
and we could use the Dragging Lemma to find points of ${\cal M}$
outside $K$ by moving $C_\ell$ into $A_{\delta/2}^+$. This proves that
${\cal M}$ has to be contained in $A_0$.
\end{proof}

In the half-space model of $\hh$ with orthonormal basis $(e_1,e_2)$
and the geodesic $\alpha_0$ represented by the half-line $\{x=0\}$,
the equidistant curves $\alpha_{-d}$ and $\alpha_{d}$ are half-lines
making with $\{x=0\}$ angles $\pm \theta$, being $\sin\theta=\tanh
d$. In this model, translations fixing the point at infinity $(0,0)$
correspond to homoteties and rotations centered at $(0,0)$. Any one of
these translations corresponds to a horizontal isometry in $\hr$ that
produces a Killing field $Y$ which is tangent to $A_0$ along the curve
$A_0 \cap (\alpha_0 \times \rr)$.

Given a point $p \in \mathcal{R}_d$ of a surface with an unit normal
vector $N(p)$ orthogonal to $Y(p)$, there exists an isometry ${\cal
  I}$ such that ${\cal I}(A_0)$ is an annulus passing through $p \in
S$, with $N(p)$ as its unit normal vector.  That isometry is nothing but a
combination of horizontal and vertical translations keeping the
boundary outside $\mathcal{R}_d$.

\medskip

Concerning the same result in $\psl$, we need the existence of a
compact stable minimal annulus $A_0$ with boundary curves outside the
vertical slab $\mathcal{R}_d$ bounded by $P^{-d}\cup P^{d}$ and the
existence of a non nulhomotopic curve $\gamma$ in $A_0$ along where
the Killing vector $Y$ is tangent to the annulus. Moreover we need to
prove that $\gamma$ is at least at distance $2d$ from the boundary
$\partial A_0$ in such a way that $\partial {\cal I}(A_0) \cap
(P^{-d}\cup P^{d}) = \emptyset$, when ${\cal I}(A_0)$ passes through
an arbitrary point $p$ of $\mathcal{R}_d$.

\begin{lemma} 
\label{pslannuli}
Given $d>0$ small enough, there exists a compact stable minimal
annulus $A_0$ in $\psl$ bounded by two large enough circles (in
exponential coordinates) $\eta_+\subset P^d$ and $\eta_{-}\subset
P^{-d}$. This annulus $A_0$ has a non nul homotopic curve $\gamma$
where the vector field $Y$ is tangent to $A_0$.  This curve $\gamma$
is at horizontal distance at least $2d$ from the boundary $\partial
A_0$.
\end{lemma}

\begin{proof}
We need to prove that, for $d >0$ small enough, there is a compact
stable annulus which is almost symmetric in an Euclidean sense. In
\cite{MCH}, the authors constructed a minimal annulus which is
asymptotic to two vertical planes $\alpha_{\epsilon}\times \rr$ and
$\alpha_{-\epsilon}\times \rr$ in $\psl$.  Let $2\epsilon$ be the
distance between these asymptotic planes and let $\alpha_0 \times \rr$
be the plane of Euclidean symmetry between the two planes. When
$\epsilon \to 0$, the set of annuli are converging to the double
covering of $\alpha_0 \times \rr$, and the waist circle of these
complete annuli is shrinking to a point. Since the curvature blows-up,
we can use an Euclidean homothety to modify the model and see forming
a convergent sequence of bounded curvature annuli which converge (see
\cite{MCH} for details) to a catenoid. Since during the process, the
plane $\alpha_0 \times \rr$ is transverse to the annuli, the limit is
a horizontal catenoid with horizontal flux. Hence for $\epsilon$ small
enough the complete annulus has almost a plane of symmetry in the
Euclidean model. We can consider a compact stable subdomain of this
example to insure that the boundary is at non zero distance from the
curve $\gamma$.
\end{proof}

\subsection{Conformal minimal immersion and finite total curvature.} 
\label{subsec:ftc}
In this section we summarize some of the results proved
in~\cite{hr,hnst}. Let ${\cal M}$ be a complete Riemann surface and
$X=(F,h):{\Sigma}\to\hr$ be a conformal minimal immersion with ${\cal
  M}=X(\Sigma)$.  We take a local conformal coordinate $z$ on
$\Sigma$. The Hopf differential associated to the harmonic map $F:
\Sigma\to\hh$ is a quadratic holomorphic differential globally defined
on $\Sigma$ and can be written as $Q=\phi(z) dz^2$. The real harmonic
function $h:\Sigma\to\rr$ can be recovered as $h=2{\rm Re }\int
-2i\sqrt{\phi}\, dz$.

If ${\cal M}$ has finite total curvature, then the immersion is proper
and by Huber's Theorem, any end $M$ of the surface ${\cal M}$ can be
conformally parameterized by ${\cal U}=\{z\in\C\ ;\ |z|\geq R\}$, for
some $R>1$ with $X({\cal U})=M$.  The Hopf differential $Q$ extends
meromorphically to the puncture $z=\infty$ and we can write
$\sqrt{\phi(z)}= \sum_{k\geq 1}a_k\,z^{-k}+P(z)$, where $a_k\in\C$ for
any $k\geq 1$ and $P$ is a polynomial of degree $m\geq 0$. We say that
the end $M$ has degree $m$.

The fact that $\phi$ extends meromorphically at the puncture implies
that the surface is transverse to any horizontal plane $\{h(z)=\pm
t\}$, for $t \geq t_0>0$ large enough and ${\cal U}\cap
X^{-1}(\{t=t_0\})$ has a finite number of connected components. The
image by $X$ of each one of these components is the boundary of a
vertical sheet $E$ contained in $X({\cal U})$.

For any conformal immersion, there is a function $\omega: \Sigma \to
\rr$ such that the third coordinate of the unit normal vector is given
by
$$n_3=\tanh \omega.$$ This function corresponds to a Jacobi field on
${\cal M}$, hence $\omega$ satisfies the differential elliptic
equation
$$\Delta_0 \omega -2 |\phi| \sinh (2\omega) = 0,$$
where $\Delta_0$ denotes the Laplacian in the Euclidean metric
$|dz|^2$. The conformal metric induced by the immersion is given by
$$ds^2= 4 \cosh^2 \omega |\phi| |dz|^2.$$
\begin{lemma} \cite{hr,hnst}
If $\phi$ is without zeroes on ${\cal U}$, the function $\omega$
satisfies the uniform decay estimate
$$|\omega (p)|Ê\leq Ce^{{\rm dist}(p, \partial {\cal U})},$$
where ${\rm dist} (p, \partial {\cal U})$ is the distance with the flat metric $|dw|^2=|\phi (z)||dz|^2$.
\end{lemma}

A finite total curvature end satisfies the hypothesis of this
lemma. If we reparametrize a vertical sheet $E$ of $X({\cal U})$ by
its third coordinate factor we obtain a conformal parameter $w=u+it$
with $\phi (w)=\frac14 (dw)^2$. In this parametrization the level set
$\Gamma_h=E \cap \{ t=t_1\}$ is a curve parametrized by $u \to
F(u,t_1)$ with geodesic curvature in $\hh$ given by (see \cite{h}):
$$k_g= \frac{-\partial_u \omega }{\cosh \omega} \hbox{ and } |\partial
_u F (u,t_1)|_{\hh}= \cosh^2 \omega,$$ and the curve $E \cap \{ u=
u_1\}$ projectes onto $\hh$ on some horizontal curve $\Gamma_v$
parametrized by $t \to F(u_1,t)$ with curvature (see \cite{h}):
$$k_g= \frac{\partial _t\omega }{\sinh \omega} \hbox{ and }
|\partial_t F (u_1,t)|_{\hh}= \sinh^2 \omega.$$ To describe the
behavior of the horizontal curves $\Gamma_h$ and $\Gamma_v$ we use the
uniform decay estimate and an interior gradient estimate
\cite{hr,hnst} to obtain
$$|\omega | (u,t) + |\nabla \omega|(u,t) Ê\leq Ce^{-t-|u|},$$ for $u
\in \rr$ and $t \geq t_1$. In particular, the level curves $\Gamma_h$
converge on compact set to horizontal geodesics at infinity while
curves $\Gamma_v$ are non proper.

\section{Characterization}

In~\cite{hr,hnst} it is proved that any complete minimal surface with
finite total curvature in $\hr$ is proper, has finite topology and
each one of its ends is asymptotic to an admissible polygon at
infinity (see Definition~\ref{def:asymp}).  In the following theorem
we prove that these conditions are not only necessary but also
sufficient.

\begin{theorem}\label{th:main}
  Let $M\subset N^3$ be a properly immersed minimal annulus with one
  boundary component and asymptotic to an admissible polygon at
  infinity. Then $M$ has finite total curvature.
\end{theorem}

\begin{corollary}
  Let ${\cal M} \subset N^3$ be a properly immersed minimal surface with
  finite topology and possibly compact boundary. Suppose that each end
  of ${\cal M}$ is asymptotic to an admissible polygon at infinity. Then ${\cal M}$
  has finite total curvature.
\end{corollary}

\begin{corollary}
  A complete minimal surface of $\hr$ has finite total curvature if an
  only if it is proper, has finite topology and each one of its ends
  is asymptotic to an admissible polygon at infinity.
\end{corollary}

The remaining part of this section is devoted to prove Theorem
\ref{th:main}.  We denote by $X:\Sigma\rightarrow N^3$ the minimal
immersion of $M=X(\Sigma)$, where $\Sigma$ is a topological
annulus. By hypothesis, $M$ is asymptotic to an admissible polygon at
infinity, denoted by~$\mathcal P$.

Let $\Omega\subset\hh$ be an open convex polygonal domain bounded by a
finite number of geodesics edges and such that the vertices of the
vertical projection of ${\cal P}$ are vertices of $\Omega$. By
possibly adding some more vertices, we can assume that $\Omega$
contains in its interior the vertical projection of $\partial M$ over
$\hh$. Using the maximum principle with vertical geodesic planes, we
get that $M\subset\Omega\times\rr$.

Up to a vertical translation, we can assume $\partial
M\subset\{t<0\}$.  Then for $t_0>0$, $M\cap\{t=t_0\}$ produces a set
of analytic curves in $\Sigma$. We will show that for $t_0>\pi$,
$X^{-1}(M\cap\{t=t_0\})$ does not contain a bounded component
$c$. Suppose it does. Then by the maximum principle using horizontal
slices, $c$ cannot bound a compact disk in $\Sigma$. Thus $c$ must be
in the homology class of $\partial\Sigma$, and $\partial\Sigma\cup c$
is the boundary of an annulus $A\subset\Sigma$, but that is not
possible by Lemma \ref{rotcatenoid}. Therefore,
$X^{-1}\left(M\cap\{t=t_0\}\right)$ cannot contain a compact
component. In particular, any component $U$ of
$X^{-1}\left(M\cap\{t>t_0\}\right)$ is simply connected and $X(U)$ is
unbounded.

Let $U$ be one such connected component and $E=X(U)$ be a vertical sheet. By
hypothesis, there exists a geodesic
$\alpha\times\{+1\}\subset{\cal P}$ such that $\partial_\infty
E\subset\partial_\infty (\alpha\times\rr)$.  We call $a_1,a_2$ the
endpoints of $\alpha$.

We now consider the minimal disks $S_h$, with $h>\pi$, introduced in
Section \ref{subsub:Sh}, associated to this geodesic $\alpha$. We
translate $S_h$ vertically upwards by an amount $t_0$.  Hence
\[
\partial_\infty S_h=(\eta \times\{t_0,t_0+h\})\cup
(\{a_1,a_2\}\times[t_0,t_0+h]),
\]
where $\eta$ is an arc in $\partial_\infty\hh$ with endpoints
$a_1,a_2$.  We recall that, when $h$ goes to $+\infty$, $S_h$
converges to the minimal graph $S_\infty$ defined on the domain of
$\hh$ bounded by $\alpha\cup\eta$ with boundary values $+\infty$ over
$\alpha$ and $t_0$ over $\eta$.  We also consider the reflected copy
$\widetilde{S}_\infty$ of $S_\infty$ across the vertical geodesic
plane $\alpha\times\rr$.

We call $P^{\epsilon}$ and $P^{-\epsilon}$ the two equidistant
vertical planes to the vertical geodesic plane $\alpha\times\rr$ at a
distance~$\epsilon$ and ${\cal R}_\epsilon$ the vertical slab bounded
by them.

\begin{claim}\label{cl:region}
  $E$ is contained in the region bounded by $S_\infty$,
  $\widetilde{S}_\infty$ and $\{t=t_0\}$. In particular, given any
  $\epsilon>0$, there exists $t_0$ sufficiently large so that $E$ is
  contained in ${\cal R}_\epsilon\cap\{t>t_0\}$.
\end{claim}

\begin{proof} 
  Let $\gamma$ be a geodesic orthogonal to $\alpha$ and consider the
  horizontal hyperbolic translations $(\varphi_s)_{s\in\rr}$ along
  $\gamma$ (being $\varphi_0$ the identity). We assume that, for any
  $s>0$, $\varphi_s(\alpha)$ has its endpoints in $\eta$.

  We fix $h>\pi$.  Since $\partial_\infty E\subset\partial_\infty
  (\alpha\times\rr)$, there exists $s_0>0$ large enough so that
  $\varphi_{s_0}(S_h)$ is disjoint from~$E$. (We observe that
  $\varphi_{s_0}(S_h)\cup\partial_\infty \varphi_{s_0}(S_h)$ is disjoint
  from $\partial E\cup\partial_\infty E$, which is contained in
  $\{t=t_0\}\cup\partial_\infty(\alpha\times\rr)$).  Now letting $s$
  decrease from $s_0$ to $0$ and using the maximum principle, we
  conclude that $S_h$ does not intersect $E$.  This holds for any
  $h>\pi$. Taking $h\to+\infty$, we conclude that $E$ lies
  below~$S_\infty$.  We finish the proof of Claim \ref{cl:region}
  following a symmetric argument using $\tilde S_\infty$.
\end{proof}

We consider the compact annulus $A_0$ presented in Section 2.3, for
some fixed $d>0$, translated so that $\alpha_0=\alpha$. We recall that
$A_0\cap(\alpha\times\rr)$ is a convex curve.  We take $\epsilon <
\ell /2$, half of the distance from $\alpha$ to the projection of
$\partial C_\ell$ onto $\hh$, for any annulus $C_\ell$ associated to
$A_0$ and~$\rho$ (we recall that $\rho$ is the radius of the extrinsic
balls where the small annuli $C_\ell$ are contained). We will fix
$\epsilon >0$ small enough in regards of $\ell$.  By Claim
\ref{cl:region}, we can assume that $E$ is contained in ${\cal
  R}_\epsilon$ for such a choice of $\epsilon$.

\begin{figure}[h]
  \centering
  \includegraphics[height=4cm]{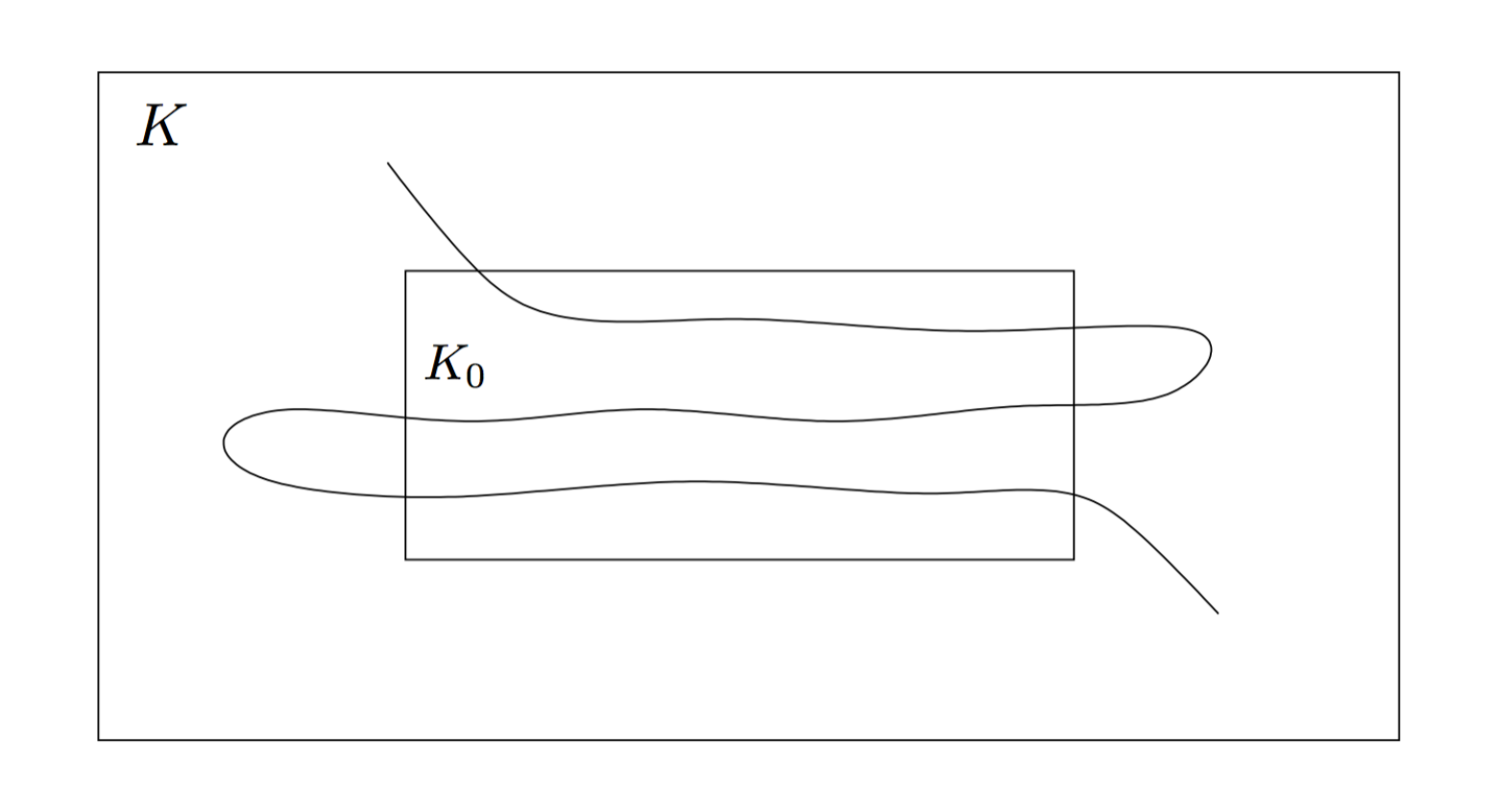}
  \caption{Compacts $K_0\subset K.$}
  \label{figure6}
\end{figure}

By properness, we can take a compact cylinder $K_0\subset {\cal
  R}_\epsilon\cap \{t\geq t_0\}$ such that $X^{-1}\left(K_0\cap
E\right)$ contains a finite number of connected components. Since any
two of these connected components can be joined by a compact arc in
$U$ and there are finitely many of those connected components, we can
find a compact set $K\subset \{t\geq t_0\}$ containing $K_0$ so that
any two points of $X^{-1}\left(K_0\cap E\right)$ can be joined by a
curve contained in $X^{-1}\left(K\cap E\right)$ (see Figure
\ref{figure6}). Moreover, if we take $K_0$ so that it contains
$C_\ell\cap {\cal R}_\epsilon$, for some small annulus $C_\ell$, then
every component of $M \cap \{ t \geq t_0\}$ has a non empty
intersection with $K_0$. If not, we can consider this annulus $C_\ell$
with $C_\ell \cap {\cal R}_\epsilon \subset K_0$ and $\partial C_\ell
\cap R_\epsilon = \emptyset$. We could move isometrically $C_\ell$
keeping its boundary outside ${\cal R}_\epsilon$ to reach any point in
${\cal R}_\epsilon$.  There will be a first point of contact with the
component that does not meet $K_0$, a contradiction with the maximum
principle. In particular we obtain:

\begin{claim}\label{cl:finite}
The number of vertical sheets of $M \cap \{ t \geq t_0\}$ is finite.
\end{claim}

\begin{figure}[h]
  \centering
  \includegraphics[height=8cm]{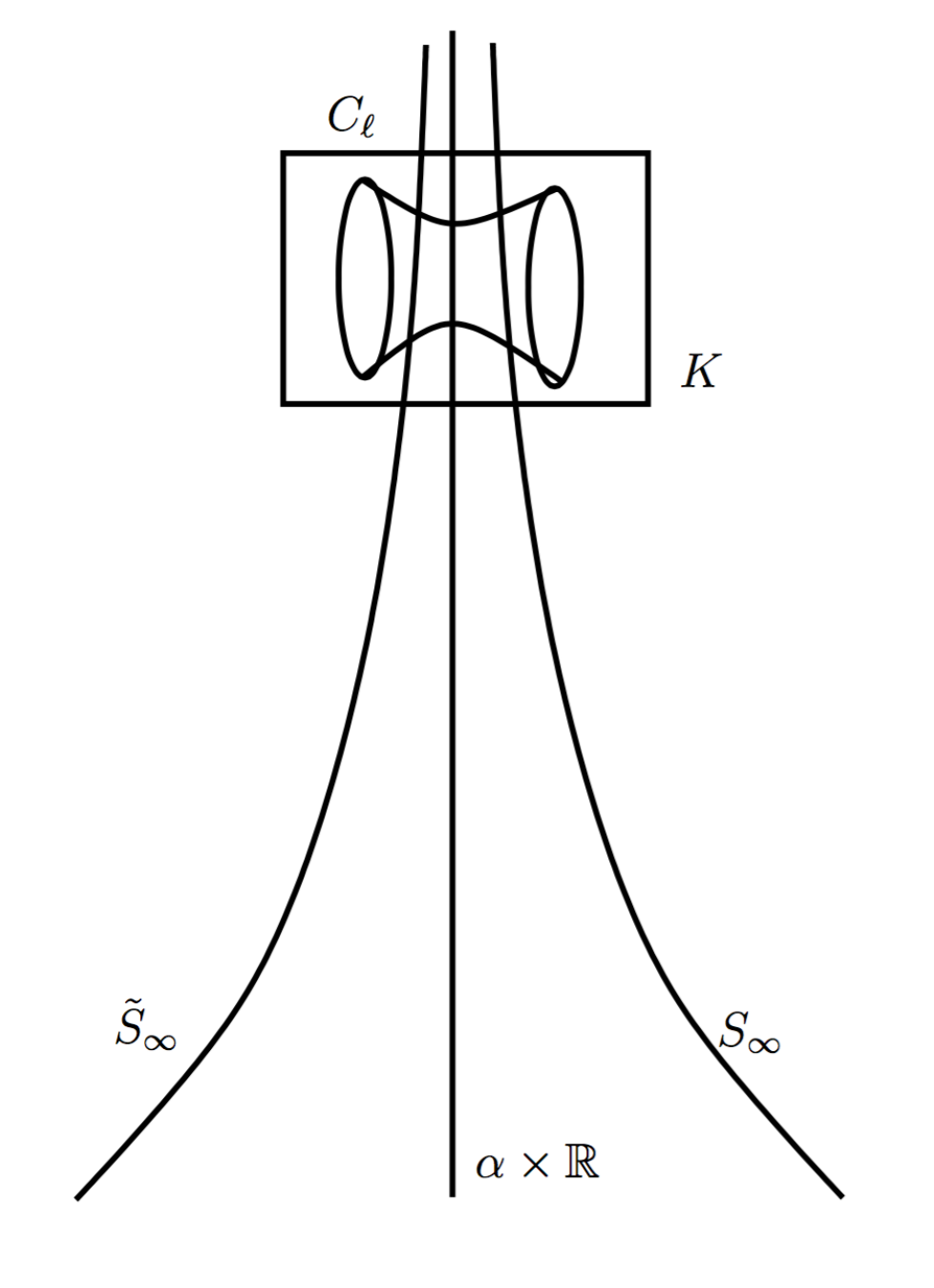}
  \caption{The slab, the compact $K$ and the annulus $C_\ell$.}
  \label{figureC}
\end{figure}

We take $t_1>t_0$ such that $K$ is contained in the slab $\{t_0\leq
t\leq t_1\}$ and let $t_2 > t_1 +2 \rho$. Now let us consider $U'$ a
connected component of $X^{-1}\left(E\cap \{t>t_2\}\right)$, and
$E'=X(U')$. Arguing as above we get that $U'$ is simply-connected and
$E'$ is unbounded.  And we can find two compact sets $K_0',K'$
contained in $\{t\geq t_2\}$ such that $K_0'\cap E'\neq\emptyset$ and
any two points of $X^{-1}\left(K_0'\cap E'\right)$ can be joined by a
curve in $X^{-1}\left(K'\cap E'\right)$. Let us consider $t_3$ such
that $K'$ is contained in the slab $\{t_2\leq t\leq t_3\}$ and take
$t_4 > t_3 +2 \rho$ (see Figure \ref{lignet}).

\medskip
\begin{figure}[h]
  \centering
  \includegraphics[height=8cm]{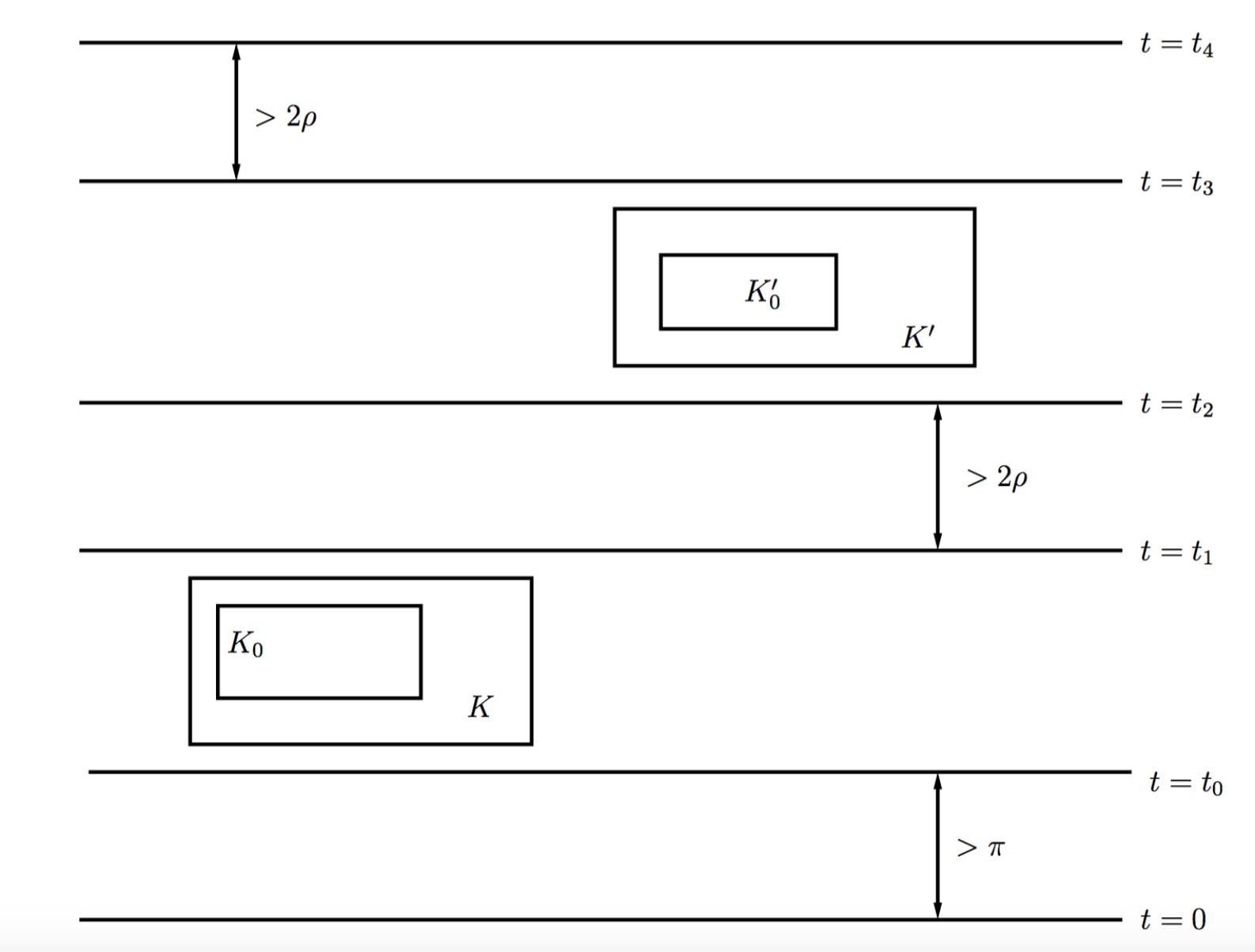}
  \caption{Compacts $K_0\subset K$ and $K'_0 \subset K'$.}
  \label{lignet}
\end{figure}

Consider $Y$ the unit vector field of $\hh$ normal to the foliation of
$\hh$ by equidistant curves to $\alpha$ (and then tangent to the
geodesics intersecting $\alpha$ orthogonally). This vector field $Y$
lifts to a horizontal field in $N^3$, also called $Y$.

\begin{claim}
 For $t'>t_4$ large, $E'\cap \{t>t'\}$ is a horizontal geodesic graph
 over a domain of $\alpha \times \rr$, i.e., $E'\cap \{t>t'\}$ is
 transversal to the horizontal vector field~$Y.$
\label{cl:graph}
\end{claim}

\begin{proof}
  Take $T>0$ so that the annulus $A_0$ is contained in a horizontal
  slab $\{ |t|Ê\leq T/2\}$.  We prove Claim \ref{cl:graph} for $t'=T +
  t_4$. Suppose by contradiction that there exists a point $z\in U'$
  such that $q=X(z)\in \{t> T + t_4 \})$ and $Y(q)\in T_q E'.$

  Choosing $\epsilon \leq \ell /2$ small enough, for any point $p\in
  A_0\cap(\gamma_0\times\rr)$ there exists an isometry $\psi_p$ of
  $\hr$ such that $\psi_p(p)=q,$ $\psi_{p}(\partial A_0)\cap {\cal
    R}_\epsilon=\emptyset$ and $\psi_{p}(\partial C_\ell)\cap {\cal
    R}_\epsilon=\emptyset$.  Hence, since the normal vector to
  $A_0$ along $A_0\cap(\gamma_0\times\rr)$ takes all directions in the
  plane $\gamma_0\times\rr$, we can find a point $p_0\in
  A_0\cap(\gamma_0\times\rr)$ such that $\psi_{p_0}(p_0)=q,$
  $\psi_{p_0}(\partial A_0)\cap {\cal R}_\epsilon=\emptyset$ and
  $\psi_{p_0}(A_0)$ is tangent to $E'$ at $q.$ In order to simplify
  the notation, we still denote by $A_0$ the annulus
  $\psi_{p_0}(A_0),$ and by $P^d$ and $P^{-d}$ the vertical planes
  containing its boundary curves $\eta$ and $\eta_{-},$
  respectively. We remark that $A_0$ is no more symmetric with respect
  to $\alpha\times\rr$ but we still denote by $\gamma_0\times\rr$ its
  plane of symmetry, so $q\in A_0\cap(\gamma_0\times\rr)$. Finally, we
  remark that $A_0\subset \{t>t_4> t_3 +2 \rho\}$, since $q\in\{t> T +
  t_4\})$.

  Since $A_0$ and $E'$ are tangent at $q$, we have that
  $X^{-1}\left(A_0\cap E'\right)$ consists, locally around $z$, of at
  least two curves passing through $z$ in an equiangular way.  Since
  neither the boundary of $A_0$ intersects $E'$ nor the boundary of
  $E'$ intersects $A_0$ and $U'$ is simply connected, we have that
  these curves bound at least two local connected components $D_1$ and
  $D_2$ in~$U'$.  These local components $D_1$ and $D_2$ are in
  distinct components of $X^{-1}(E' \cap A_0^-)$.  In fact, suppose
  this is not true, so we can find a path $\alpha_0$ in $U'$ with
  $X(\alpha_0) \subset A_0^-$, joining points $x \in D_1$ and $y \in
  D_2 $. Now join $x$ to $y$ by a local path $\beta_0$ in $U'$ going
  through $p$ with $\beta_0 \subset A_0^-$ except at $p$. Let $\Gamma
  = \alpha_0 \cup \beta_0$, we have $X(\Gamma) \subset A_0^-$. Since
  $U'$ is simply connected, $\Gamma$ bounds a disk $D$ in $U'$, but by
  construction $X(D) \cap A_0 ^+ \neq \emptyset$, a contradiction with
  Lemma \ref{lema}. Hence $D_1$ and $D_2$ are contained in two
  distinct components of $U'\backslash X^{-1} (A_0 \cap E')$ such that
  $X(D_1)$ and $X(D_2)$ are contained in $E' \cap A_0^-$ (see Figure
  \ref{lignetD}).

 We observe that $D_1$ and $D_2$ are disjoint in $U'$, even if their
 images by $X$ intersect each other (as the surface $M$ is not
 necessarily embedded).

\begin{figure}[h]
  \centering
  \includegraphics[height=6cm]{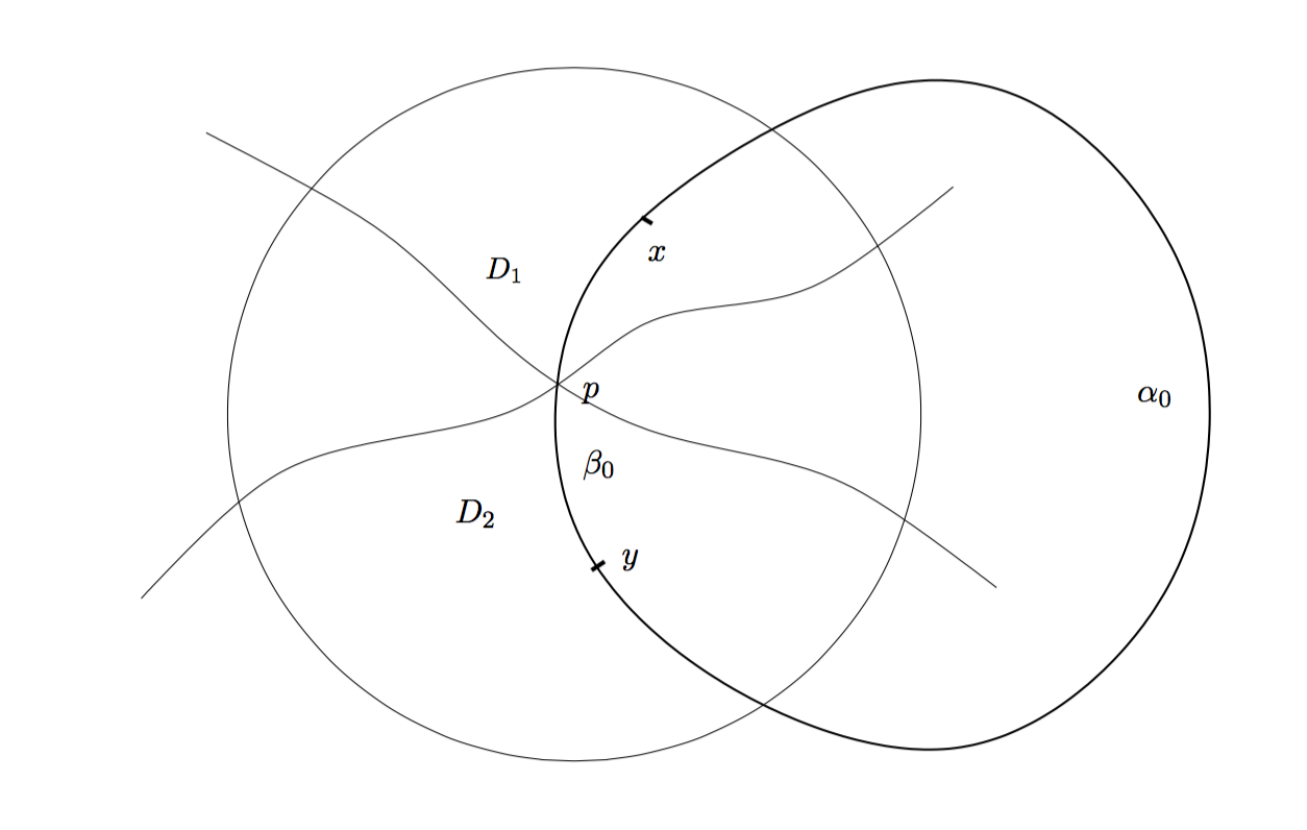}
  \caption{Components $D_1$ and $D_2$}
  \label{lignetD}
\end{figure}

  Let $P$ be a vertical geodesic plane orthogonal to
  $\gamma_0\times\rr$ such that $P$ divides the intersection curve
  $A_{-\delta}\cap(\gamma_0\times\rr)$ in two components. Denote by
  $P^{-}$ and $P^{+}$ the two halfspaces determined by~$P$, and by
  $P^{-\rho}, P^{\rho}$ the two equidistant planes to $P$ at a
  distance~$2\rho$ from~$P$. Let us denote by $[P^{-\rho}, P^{\rho}]$
  the slab bounded by $P^{-\rho}$ and $P^{\rho}$.  We consider the
  ball $B_0=B_{q_1}(\rho)$ in $A_0^-$ of radius $\rho$ centered at the
  point $q_1$ of the ``axis" $\xi$ of the annulus $A_0$ (see Section
  2.3). This ball $B_0$ contains an annulus $C_{\ell_0}$ with boundary
  outside ${\cal R}_\epsilon.$ Notice that $B_{q_1}(2\rho)\cap
  \mbox{Tub}^{-}(A_0)=\emptyset$ and then $B_{q_1}(4\rho)\subset
  A_0^-$. $B_0$ is contained in $[P^{-\rho}, P^{\rho}]$ with its
  center at height $t=s_0$.  For any $j=1,2$, we are going to prove
  that $X(D_j)$ contains a point $z_j$ inside $C_{\ell_0} \subset
  B_0$. To see that we apply the maximum principle.  We know that
  $X(D_j)$ intersects each annulus $A_s$ in Tub$^{-}(A_0)$; then
  $X(D_j)$ intersects a small annulus $C_{\ell_j}$ in Tub$^{-}(A_0)$
  at a point $w_j$.  By our choice of $\epsilon$, the boundary of
  $C_{\ell_j}$ does not intersect $X(D_j)$ when translated in the
  direction of a vector in $\alpha\times\rr$.  Using the Dragging
  Lemma with $D_j\subset U'$ and translated copies of $C_{\ell_j}$, we
  find a curve in $X(D_j)$ going from the point $w_j$ to a point $z_j
  \in C_{\ell_0}$, as desired. See Figure~\ref{mexicain}.
  
\begin{figure}[h]
  \centering
  \includegraphics[height=9cm]{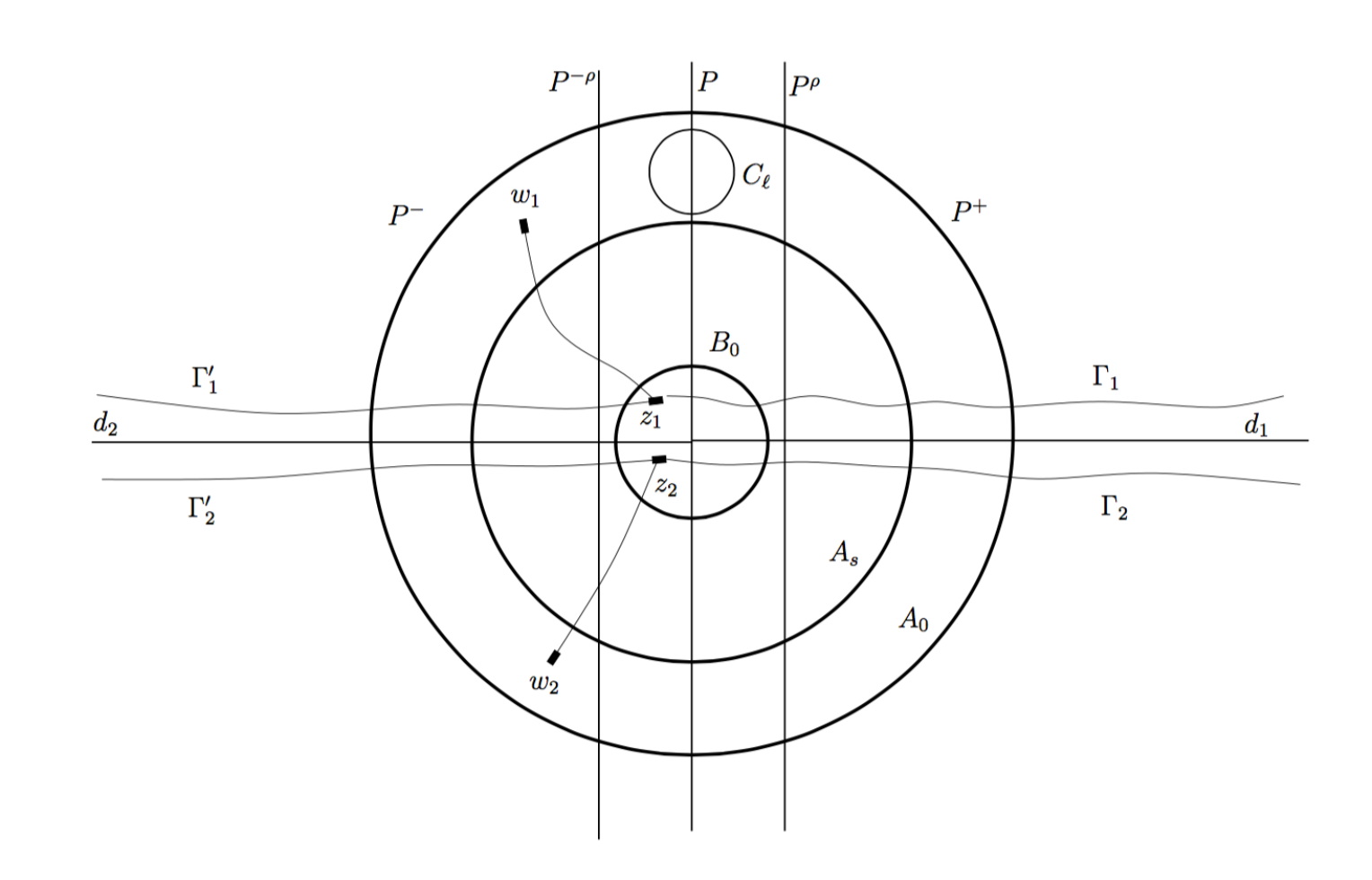}
  \caption{A vertical section of $A_0, C_\ell$ and $[P^{-2\rho},
      P^{2\rho}].$}
  \label{mexicain}
\end{figure}

Using again the Dragging Lemma we consider horizontal translations of
$C_{\ell_0},$ denoted by $C_{\ell} (t)$, along a horizontal geodesic
in $P^+,$ going very far from the slab
$[P^{-\rho},P^{\rho}],$ and then going vertically downwards into the
compact $K_0'$ (applying as well horizontal translations, if needed,
in order to get to the compact $K_0'$ but with $\partial C_{\ell} (t)$
never touching the slab).  Along this movement we follow a connected
arc $\Gamma'_1 (t) \in U'$ with $X(\Gamma'_1 (t)) \in C_{\ell} (t)$
starting at $z_1$ and ending into $K_0'$, with $X(\Gamma' _1) \cap
A_0^- \subset X(D_1)$.  We denote by $\Gamma' _2 (t) \subset U'$ a
similar arc starting at $z_2$ and ending into $K_0'$ with $X(\Gamma'
_2) \cap A_0^- \subset X(D_2)$. Notice we cannot connect $\Gamma'_1$
and $\Gamma'_2$ before they leave $A_0^-$ (since $D_1$ and $D_2$ are
in two disjoint components). If they meet in $U'$ before ending into
$K_0'$, we stop the construction at the first point of intersection.
If not we connect $\Gamma'_1$ and $\Gamma'_2$ by an arc $\alpha'$
contained in the compact set $X^{-1}(K')$, since they have endpoint in
$K_0'$ (see Figure \ref{mexicain}). 

We apply the same construction in $P^-$. We move the annulus
$C_{\ell}$ along first a horizontal geodesic into $P^-$, next we move
the annulus vertically downwards (and horizontally, if necessary) to
end into $K_0$. We construct an arc $\Gamma _1$ from $z_1$ to $K_0$
and an arc $\Gamma_2$ from $z_2$ into $K_0$. We connect eventually
$\Gamma_1$ and $\Gamma_2$ by an arc $\alpha$ in $X^{-1}(K)$.

 \begin{figure}[h]
    \centering
    \includegraphics[height=12cm]{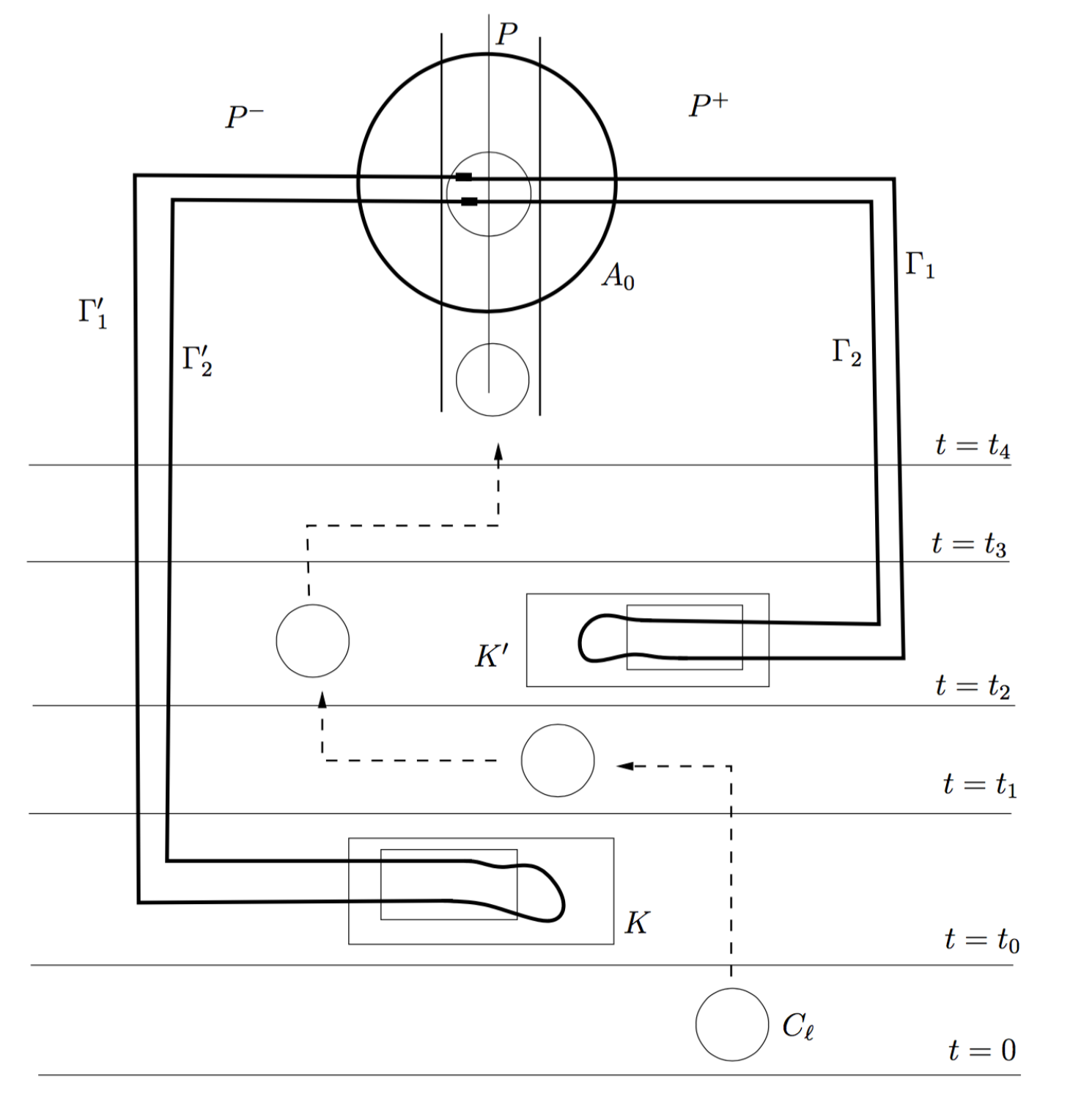}
    \caption{Curve $\widetilde\Gamma$.}
    \label{figure7}
  \end{figure}

  Therefore $\widetilde{\Gamma}=\Gamma_1\cup \Gamma_2\cup \alpha \cup
   \Gamma'_1\cup \Gamma_2' \cap \alpha '$ is
  a simple closed curve in $U$ (see Figure \ref{figure7}). We call
  $D$ the compact disk in $U$ bounded by $\widetilde{\Gamma}$.
  Now we consider a small annulus $C_{\ell}\subset \mbox{Tub}^{+}(A_0)$
  lying under~$A_0$ and such that $C_{\ell}\cap P\neq \emptyset$. For
  instance, we can take $C_{\ell}$ contained in the ball of radius $\rho$
  centered at a point in $P\cap A_{\delta/2}\cap(\gamma_0\times\rr).$

  We move isometrically the small annulus $C_{\ell}$ to $\{t<t_0\}$
  keeping its boundary outside ${\cal R}_\epsilon$. We consider a
  family of annuli obtained by continuously translating back
  $C_{\ell}$ from $\{t<t_0\}$ to its original position in
  $\mbox{Tub}^{+}(A_0) \cap [P^{-\rho},P^{\rho}]$. By the choice of
  $\widetilde{\Gamma}$, we can assume that none of these translated
  annuli intersects $X(\widetilde{\Gamma})$ and then, by the maximum
  principle, they do not intersect the minimal disk $X(D)$ either
  (observe that $X(D)\subset{\cal R}_\epsilon$, so $X(D)$ cannot
  intersect the boundary of the annuli). Translating slightly
  $C_{\ell}$ inside $A_0^-$ and using the maximum principle, we get
  that $X(D)$ has no points in $P\cap \{t_4\leq t\leq s_0-3\rho\}.$
  Using the maximum
  principle with the anulli $C_\ell$ coming from above $A_0$ and going
  downwards inside $A_0^-$ in the region
  $[P^{-\rho},P^{\rho}]\cap\{t\geq s_0+3\rho\},$ we also conclude that
  $X(D)$ has no points in $P\cap \{ t\geq s_0+ 3\rho\}.$
  Possibly slightly translating $P$, we can suppose that $X(D)\cap P$ is
  transversal.

  Since $X(\Gamma_1 \cup \Gamma_1')$ is a curve which crosses $P$
  transversally going from $P^-$ to $P^+$ then the number of points in
  $(\Gamma_1\cup \Gamma_1' )\cap X^{-1}(P)$ is odd.  Now we consider a
  curve $\beta$ in $D\cap X^{-1}(P)$ starting at a point $p\in
  \Gamma_1\cup \Gamma_1'$. We observe that by the maximum principle
  $\beta$ can not be a closed curve, so it does not finish at $p$.
  Since $X(D)$ does not intersect $P\cap (\{t_4\leq t\leq
  s_0-3\rho\}\cup\{t\geq s_0+3\rho\})$, we get that $\beta$ is
  contained in $X^{-1}(A_0^{-}\cap P).$ Hence $\beta$ is a curve
  entirely contained in the component $D_1$ and finishes at a
  different point in $\Gamma_1 \cup \Gamma_1'$ concluding that the
  number of points in the intersection $P \cap (\Gamma_1 \cup
  \Gamma_1')$ is even, a contradiction.  Therefore, we conclude that
  $E'\cap \{t\geq t'=T +t_4 \}$ is necessarily transversal to the
  horizontal vector field $Y.$
\end{proof}

Let $t'>T+t_4$, we call $\widetilde{U}$
a connected component of $X^{-1}(E'\cap \{t>t'\})$ and
$\widetilde{E}=X(\widetilde{U}).$ We have proved that the horizontal sheet
$\widetilde{E}$ is a horizontal geodesic graph over $\alpha \times \rr$ for some geodesic in the direction of the vector field $Y$.

Similarly, we can prove that for $t'$ large, $X^{-1}(M\cap\{t<-t'\})$
has a finite number of connected components $\tilde V$ and each
$X(\tilde V)$ is a horizontal graph over $\beta\times\rr$ for some
geodesic $\beta$ such that $\beta\times\{-1\}\subset{\cal P}$.

\begin{claim}
For $t'>T+t_4$ large, $E'\cap \{t>t'\}$ (resp. $ E'\cap \{t<-t'\}$) is
a horizontal Killing graph over a domain of $\alpha \times \rr$
(resp. $\beta \times \rr$).
\label{cl:killinggraph}
\end{claim}

\begin{proof}
We consider the half-space model of $\hh$ with orthonormal basis
$(e_1,e_2)$ and the geodesic $\alpha$ represented by the half-line
$\{x=0\}$. In this model the equidistant curves $\alpha^{-\epsilon}$
and $\alpha^{\epsilon}$ are half-lines making angles $\pm \theta$ with
$\{x=0\}$. In this model, horizontal translations of the annulus $A_0$
(keeping the origin as a fixed point at infinity) correspond to
homoteties and rotations centered at the origin of the model in such a
way that the boundary curves of $A_0$ do not intersect
$\alpha^{-\epsilon}\times \rr$ nor $\alpha^{\epsilon} \times
\rr$. Claim 3 implies that we cannot find a point of $E$ which is
tangent to $A_0$ along its symmetry curve $\gamma_0 \times \rr$.
Looking for the set of admissible transformations, Claim 3 concludes
that $E$ is transverse to any geodesic orthogonal to $\alpha \times
\rr$.

To prove that the component $E$ is a horizontal Killing field, it
suffices to move the annulus $A_0$ with homothety and horizontal
translation along $e_1$ direction and check that we can place
$\gamma_0 \times \rr$ in any position of the slab with boundary of
$A_0$ not intersecting $\alpha^{-\epsilon}\times \rr$ and
$\alpha^{\epsilon} \times \rr$. If we choose $\epsilon >0$ such that
it is possible at height $y=1$ into the half-plane model, then we have
the same degree of freedom at each $y >1$, using isometries of $\hh$.
\end{proof}

Let $\alpha\times\{1\}$ and $\beta\times\{-1\}$ be two
geodesics in ${\cal P}$ such that $\alpha$ and $\beta$ share an
endpoint $a\in\partial_\infty\hh$. We consider a foliation of
horocylinders $\bigcup_{c\geq c_0}{\mathcal H}(c)$ with boundary points
at infinity $\{a \} \times \rr$ and we consider a horizontal sheet $E$
of $M\cap\bigcup_{c\geq c_0}{\mathcal H}(c)$ parametrized by
$X(U)$. Recall that $X(U) \subset \Omega \times \rr$ (see the
beginning of Section 3). Let us denote by $\alpha_1,\beta_1$ the
geodesics in $\partial \Omega$.  Let $\gamma_c$ be the geodesic
orthogonal to $\alpha$ passing through the point $\alpha \cap
{\mathcal H}(c)$. For $c$ large enough,
$\beta_1\cap\gamma_c$ is non empty.

We consider the cusp end of $\Omega$ bounded by arcs $\tilde \gamma_{c_0},
\tilde \alpha_1$ and $\tilde \beta_1$, contained in $\gamma_{c_0},
\alpha_1$ and $\beta_1$. Consider $c_0$ sufficiently large so that the
distance between $\gamma_{c_0}\cap \tilde \alpha_1$ and
$\gamma_{c_0}\cap \tilde \beta_1$ is less than $\epsilon$ and $M \cap
\bigcup_{c\geq c_0}{\mathcal H}(c)$ is contained in a vertical slab
${\cal R}_\epsilon$
of width $\epsilon$.

Let $Y$ be the vector field defined just before Claim \ref{cl:graph}
for $\alpha$.

\begin{claim}
  For $c'$ large, any horizontal sheet in $M\cap \bigcup_{c\geq c_0}{\mathcal H}(c)$ is a horizontal geodesic graph over a domain of $\alpha
  \times \rr$ transverse to the field $Y$ and contained in
  ${\cal R}_{\epsilon}$.

  \label{cl:hor}
\end{claim}
\begin{proof}
  We take $c_0\in\rr$ large enough so that $\partial M \cap
  \left(\cup_{c\geq c_0}\gamma_c\times\rr\right) =\emptyset$.  By the
    maximum principle using vertical geodesic planes, we know that no
    connected component of $M\cap(\gamma_{c_0}\times\rr)$ can bound a
    compact disk in $M$.

We take an ideal geodesic quadrilateral ${\cal Q}\subset\hh$, two of
whose opposite edges are $\gamma_{c_0}$ and $\gamma_{c_1}$, for some
$c_1>c_0$, and such that there exists a Scherk minimal graph defined
over ${\cal Q}$ with boundary values $+\infty$ over
$\gamma_{c_0}\cup\gamma_{c_1}$ and $-\infty$ over the other two
edges. Taking $c_0$ large enough, we can assume that $\partial{\cal
  Q}\setminus(\gamma_{c_0}\cup\gamma_{c_1})\subset\hh\setminus\Omega$
(recall that $M\subset\Omega\times\rr$).

  We now claim that $M\cap(\gamma_{c}\times\rr),$ for $c> c_1$, does
  not contain any compact curve $\Gamma$. Suppose this was not the
  case. We already know that $\Gamma$ has to be in the homology class
  of $\partial M$. Thus $\Gamma \cup\partial M$ bounds an annulus $A$.
  Using the maximum principle with $A$ and vertically translated
  copies of the Scherk graph just described above, we reach a
  contradiction.

  Given $c_2>c_1$, we consider a connected component $E=X(U)$, where $U$ is a connected
  component of  $X^{-1}\left(M\cap(\cup_{c> c_2}\gamma_c\times\rr)\right)$.  We can assume that
  $(\gamma_{c_2}\times\rr) \cap {\cal R}_{\epsilon}$ is
  transversal. We have proved that $U$ is simply-connected and
  $\partial E$ consists of curves joining $( \alpha\cap
  \gamma_{c_2})\times\{-1\}$ to $( \beta \cap
  \gamma_{c_2})\times\{1\}$ (possibly no one) and perhaps some
  curves whose endpoints are both in $(\beta\cap
  \gamma_{c_2})\times\{-1\}$ or $( \alpha\cap \gamma_{c_2})
  \times\{1\}$.
  We know that $X^{-1}(M\cap \{|t|>t'\})$ has a finite number of
  connected components and each one of them corresponds to (via $X$) a
  horizontal graph. Hence we get that $\partial E$ has a finite number
  of curves.

  Take a compact set $K_0$ contained in the halfspace $\cup_{c>
    c_2} (\gamma_c\times\rr)$.  By properness, there are a finite number
  of connected components in $X^{-1}(K_0\cap E)$. Hence we can find a
  compact set $K\subset\cup_{c> c_2} (\gamma_c\times\rr)$ containing
  $K_0$ so that any two points of $X^{-1}(K_0\cap E)$ can be joined by
  a curve contained in $X^{-1}(K\cap E)$. Let us suppose that $K$ is
  contained in the vertical slab between $(\gamma_{c_2}\times\rr)$ and
  $(\gamma_{c_3}\times\rr),$ for some $c_3>c_2.$

  Now let us consider a connected component $E'=X(U')$, where $U'$ is a connected component
  of $X^{-1}\left(E\cap (\cup_{\{c\geq
    c_3+2\rho\}}\gamma_c\times\rr)\right),$
  where $\rho$ is the radius of the extrinsic balls where the small
  annuli $C_\ell$ are contained in. We observe that the boundary of
  $\partial C_\ell$ is outside ${\cal R}_\epsilon$.

  Analogously, we can find two compact sets
  $K'_0,K'\subset\cup_{\{c\geq c_3+2\rho\}}\gamma_c\times\rr$ such
  that any two points of $X^{-1}(K'_0\cap E')$ can be joined by a
  curve in $X^{-1}(K'\cap E').$ We take $c_4>c_3+2\rho$ such that $K'$
  is contained in the vertical slab between
  $\gamma_{c_3+2\rho}\times\rr$ and $\gamma_{c_4}\times\rr$.

  Suppose that the annulus $A_0$ can be contained in $\cup_{\{0\leq
    r\leq c_5\}}(\gamma_{r}\times\rr)$ and we take $c'>c_4+c_5+2\rho$.
  Now we can argue verbatim to the proof of Claim \ref{cl:graph} using
  this new $E'$ contained in the slab ${\cal R}_\epsilon$ and the
  annuli $A_0$ and $C_\ell$ with boundaries outside ${\cal
    R}_\epsilon$. To do that it is enough to check that we can move
  $A_0$ and $C_\ell$ with boundary curves outside ${\cal
    R}_\epsilon$. In the model of the half-plane with $a$ at infinity,
  the curves $\tilde \alpha_1$ and $\tilde \beta_1$ are parallel
  vertical half-lines. The annuli can be moved using homothety and
  horizontal translations as in Claim \ref{cl:killinggraph}.
  
  These arguments show that $E'\cap\cup_{c\geq c'}
  (\gamma_c\times\rr)$ is transversal to the vector field $Y$.

  We call $\widetilde U$ a connected component of $X^{-1}\left(E'\cap
  \cup_{\{c\geq c'\}}(\gamma_c\times\rr)\right)$ and $\widetilde
  E=X(\widetilde U)$. Hence $\widetilde E$ is a horizontal geodesic
  and Killing graph. More precisely, for any $c>c'$, $\widetilde E\cap
  (\gamma_c\times\rr)$ consists of one curve $d_c=d_c(t)$ joining
  $(\beta\cap \gamma_{c'})\times\{-1\}$ to $(\alpha\cap
  \gamma_{c'})\times\{1\}$ and $\partial \widetilde E=d_{c'}$.

\end{proof}

\begin{claim}
A horizontal Killing graph asymptotic to an admissible polygonal has
finite total curvature

  \label{cl:ftc}
\end{claim}

 \begin{figure}[h]
    \centering
    \includegraphics[height=8cm]{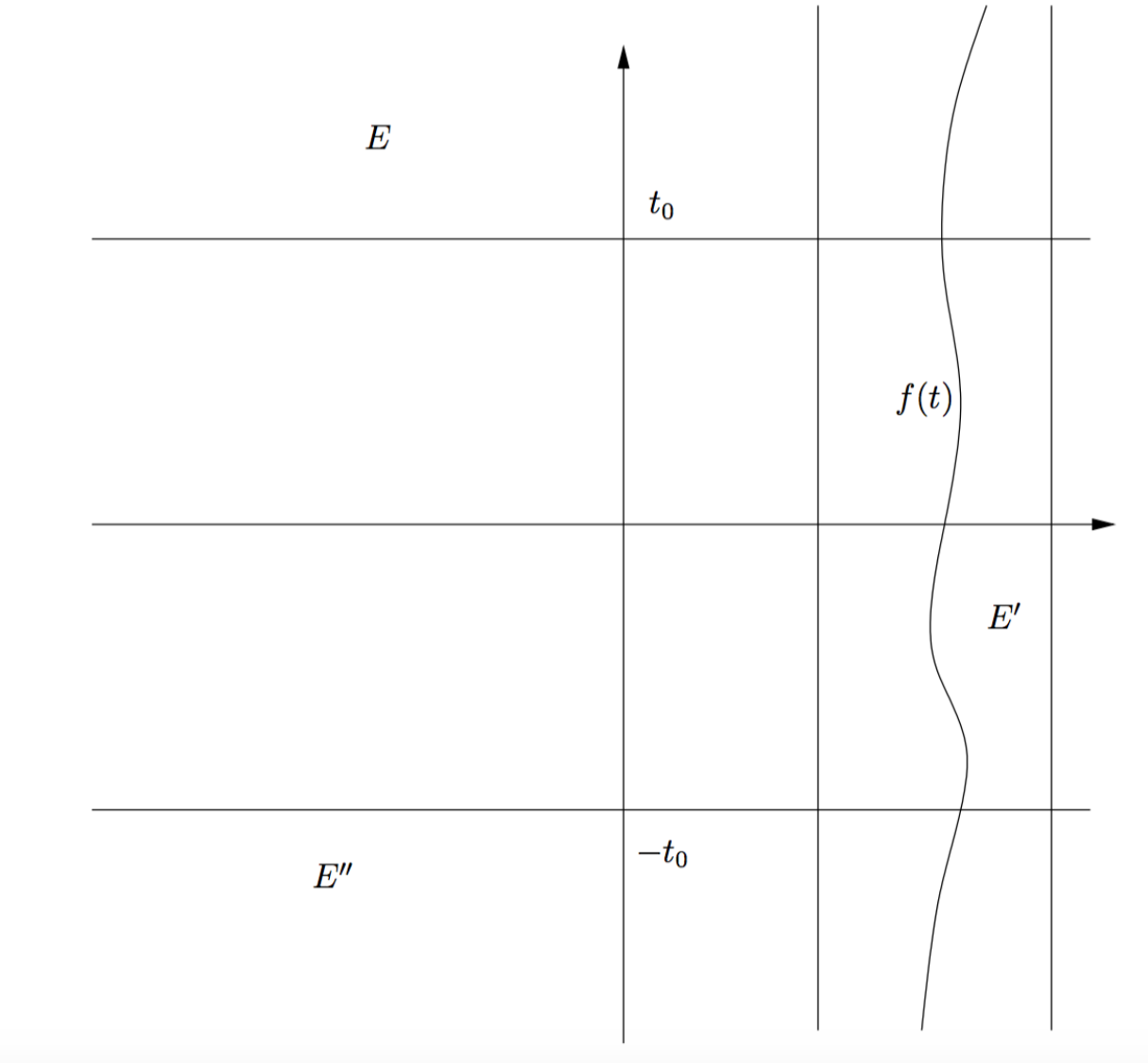}
    \caption{Domain $E \cup E' \cup E''$}
    \label{fig:lastfig}
  \end{figure}

\begin{proof}
In $\psl$ this claim has been proved in \cite{mn} by constructing
barrier to control the decay of the curvature $K$ on a Killing
horizontal multigraph with finite number of sheets. Killing fields are
used to get estimates of curvature by stability argument. The number
of sheets is limited by the boundary at infinity of $M$ which is a
polygonal line with finite number of edges and the fact that we assume
that the immersion is proper.

In $\hh \times \rr$, we complete the proof by using conformal
parametrization by its third coordinate.  We consider a vertical sheet
$E \subset M \cap \{t > t_0\}$ and we reparametrize conformally $E$ by
$w=u+it$ on $\Omega_0 = \{ t \geq t_0\},$ where $t$ is the third
coordinate of the conformal immersion in $\hh \times \rr$ (since $E$
is a horizontal graph, the tangent plane is never horizontal and $t$
is defined globally on $E$) given by
$$X(u,t)= (F(u,t),t) \in \hh \times \rr.$$

The quadratic Hopf differential associate with the horizontal harmonic map $F : \Omega \to \hh$ is given
by $Q=\frac{1}{4} (dw)^2$ and the conformal metric of the conformal immersion is given by
$ds^2 = \cosh^2 \omega |dw|^2$ with 
\begin{equation}
|F_u|_{\hh}= \cosh ^2 \omega \hbox{ and } |F_t|_{\hh}= \sinh ^2 \omega.
\label{conformal}
\end{equation}
In this parametrization $n_3 = \tanh \omega$ is the third component of the normal (see Section 2.4). The decay estimate on $\omega$ providing $Q$ has no zeros on $E$ is given by
$$|\omega| (u,t) \leq C_0e^{-t},$$ 
where $C_0$ is a constant depending on $t_1$, for any $t \geq t_1
>t_0$, $t_1$ large enough. We need to improve this decay in the
variable $u$ to conclude finite total curvature.

Using barriers, $E$ is uniformly asymptotic to $\alpha \times \rr$ at
infinity. Let $a \in \partial _{\infty} \alpha$ and consider $M \cap
\{ {\mathcal H} (c); c \geq c_0\}$ where $\partial _\infty
{\mathcal H} (c) = \{a \}\times\rr$. There is a horizontal sheet $E' \subset
MÊ\cap \{ {\mathcal H} (c); c \geq c_0\}$ which extends to
$E$.

Since $E'$ is a horizontal multigraph on $\alpha \times \rr$, $E'$ can
be conformally parametrized by some topological half-plane $\{w=u+it;
u \geq f_0(t)\},$ where $f_0 : \rr \to \rr$ is a continuous function.
The curve $t \to X(f(t),t)$ is a parametrization of $\partial E'
\subset M \cap {\mathcal H} (c_0)$ which extends $E$. For $t \leq
-t_0$, $E'$ is connected to a vertical sheet $E'' \subset M \cap \{ t
\leq -t_0\}$.

We parametrize $E \cup E' \cup E''$ globally on some domain $\Omega_1
\subset \C$ where the third coordinate is $x_3=t$ (see Figure
\ref{fig:lastfig}). We study $t \to F(s,t) \subset \hh$, the immersion of
the curve $\{u=s; t \in \rr\}$ for $s \geq \sup_{t \in [-t_0,t_0]} f
(t)$. Since $|\omega|(s,t) \leq C'e^{-t}$, using (\ref{conformal}),
the curve $t \to F(s,t)$ is non proper and contained in a compact arc
linking two points into consecutive geodesics $\alpha$ and $\beta$ of
${\mathcal P}$ having same infinite point $a \in \partial _\infty \hh$. 
When $s \to \infty$, the arc is uniformly
diverging to infinity since the end $M$ is properly immersed. In
particular, for $s$ large enough, the arc is contained in the convex
domain bounded by ${\mathcal H} (c_1)$ for some $c_1 \geq c_0$.
We observe that this argument implies that for any $c_2 \geq c_1$, the
curve $(E \cup E' \cup E'')\cap {\mathcal H} (c_2)$ is parametrized in
$\Omega$ by some curve $\{ u =f_2 (t)\}$ with $C_1 \leq f_2 (t) \leq
C_2$.

This estimate provides a control of the boundary of horizontal sheet in the parametrization by the third
coordinate on $\Omega_1$. Now we apply the decay estimate on $E'$ and obtain that
$$|\omega | (u,t) \leq C_1 e^{-u},$$
for $u > c_3$. This estimate holds on points of the vertical sheet $E$ for $u> c_3$ large enough close to the boundary  point $a \in \partial_\infty \alpha$. We can do the same argument at the other boundary point
of $\alpha$. Doing this we obtain a uniform decay on any vertical sheet $E$:

$$|\omega | (u,t) \leq C_2 e^{-t - |u|}, \hbox{ for } t \geq t_1 >t_0 \hbox{ and } u \in \rr.$$
Now using the fact that $\omega$ is solution of the elliptic equation $\Delta_0 \omega -2|\phi |\sinhÊ\omega =0$,
we apply interior gradient estimates to obtain that
$$|\omega|(u,t) + |\nabla \omega | (u,t) \leq C_3 e^{-t - |u|} , \hbox{ for } t \geq t_1 >t_0 \hbox{ and } u \in \rr.$$
This estimate provides finite total curvature for a vertical sheet $E$. The metric is $ds^2 =Ê\cosh^2 \omega |dz|^2$
and the curvature is given by
$$K(u,t)= -\tanh^2 \omega - \frac{|\nabla \omega|^2}{4 \cosh^2 \omega} \leq 0.$$
Using the exponential decay this proves that any vertical sheet of $M \cap \{ |t |\geqÊt_1\}$ has finite total curvature:

$$\int _{E \cap \{t \geq t_1\}Ê} |K| dA Ê\leq C.$$

It remains to study the horizontal sheets of $M$ in a slab which are connecting two vertical sheets as we describe above. We parametrize any horizontal sheet $E' \subset M \cap \{|t|\leq 2t_1 \}$ and we use the decay in the variable
$u$ with variable $t$ bounded to obtain finite total curvature. Since the number of these sheets is bounded,
the end $M$ has finite total curvature.

\end{proof}

For the next result we are going to use the following notion of length of a complete geodesic. Let $\Gamma$ be an ideal polygon with geodesics $\alpha_i,\beta_i$. At each vertex $a_i$ of $\Gamma$ consider a horocycle $H_i$ so that $H_i\cap H_j=\emptyset$ if $i\neq j.$ Each geodesic $\alpha_i$ meets two of these horocycles, we denote by $|\alpha_i|$ the distance between them. We define $|\beta_i|$ in the same way. It is easy to see that $\sum_{i=1}^k |\alpha_i|- \sum_{i=1}^k|\beta_i|$ does not depend on the choice of the disjoint horocycles.

As corollaries of Theorem \ref{th:main} we obtain the following
theorems.

\begin{theorem}
Let $M$ be a properly embedded minimal disk in $\hr$ asymptotic to
an admissible polygon at infinity ${\cal P}$. Suppose that the
vertical projection of ${\cal P}$ in $\hh$ is the boundary of a
convex domain $\Omega$. Then $M$ is a vertical graph.

In particular, if $\alpha_i\times\{1\}$ and
$\beta_i\times\{-1\}$, with $i=1,\dots, k$, are the edges of
${\cal P}$ then:
\begin{enumerate}
\item $\sum_{i=1}^k |\alpha_i|= \sum_{i=1}^k |\beta_i|$; and
 \item for any inscribed polygonal domain $D$ in $\Omega$,
$\sum_{i=1}^k |\alpha_i\cap\partial D|= \sum_{i=1}^k
 |\beta_i\cap\partial D|$,
\end{enumerate}
 where $|\bullet|$ denotes the hyperbolic length of the curve
 $\bullet$ as above.
\end{theorem}

\begin{proof}
We first observe that, using the maximum principle with vertical
geodesic planes, we get that $M\subset\Omega\times\rr$.

By Claims \ref{cl:finite} and \ref{cl:graph}, there exists $t'\in\rr$ such that
$M\cap\{t> t'\}$ has a finite number of connected components (we are
using that $M$ is embedded), being each one of them a horizontal
geodesic graph over the plane defined by a horizontal geodesic in
${\cal P}$.  Let $E$ be one such component that is a horizontal
geodesic graph over a domain of $\alpha\times\rr$, where
$\alpha\times\{1\}\subset{\cal P}$. We call $E^*$ the reflected copy
of $E$ with respect to $\{t=t'\}$.

Let $\Delta$ be the component of $\h-\alpha$ such that
$\Delta\cap\Omega =\emptyset$, and fix a point
$p_\infty\in\partial_\infty\Delta-\overline\alpha$. We take a geodesic
$\gamma$ orthogonal to $\alpha$ with $p_\infty$ as one of its
endpoints. We can consider a horizontal translation of $E^*$ along
$\gamma$ towards $p_\infty\times\rr$ so that it does not intersect
$M$. We translate back $E^*$ until reaching its original position. By
the maximum principle, none of these translated copies of $E^*$ can
intersect $M$.

 Up to a vertical translation, we can assume $t'=0$. We denote
 $M_s=M\cap\{t<-s\}$ and $M_s^*$ the reflected copy of
 $M-\overline{M_s}$ with respect to $\{t=-s\}$. We have proved that
 $M_0$ and $M_0^*$ are disjoint. We can then start applying Alexandrov
 Reflection Principle, and we obtain that $M_s$ and $M_s^*$ are
 disjoint for any $s>0$.  Since $t'$ could be taken arbitrarily large,
 we conclude that $M$ is a vertical graph.

 Since $M$ is a vertical graph assuming $+\infty$ on $\alpha_i$ and $-\infty$ on $\beta_i,$ we know that it has to be a Jenkins-Serrin-type graph. In particular, by Collin and Rosenberg \cite{CR} the length of the geodesic arcs on the ideal polygon must satisfy the conditions described on the statement of this theorem.
  
\end{proof}

\begin{theorem}
Let $M$ be a properly embedded minimal surface in $\hr$ asymptotic to
a finite number of vertical geodesic planes $\alpha_i\times\rr$,
$i=1,\dots,k$, cyclically ordered (i.e. there exists an ideal convex
polygonal domain $\Omega\subset\h$ whose vertices - all of them at
infinity - are the endpoints of the geodesics $\alpha_i$.) Then $M$ is
a vertical bigraph symmetric with respect to a horizontal slice.
\end{theorem}

\begin{proof}
We proceed as in the proof of the previous theorem and obtain that
$M_0$ and $M_0^*$ are disjoint. More precisely, if ${\cal M}$ denotes
the component of $\hr-M$ containing $\Delta\times\rr$, then $M_0^*$ is
contained in ${\cal M}$. Applying Alexandrov Reflection Principle we
get that $M_s^*\subset{\cal M}$ for $s$ small. But there must exist
$s_0>0$ such that $M_{s_0}^*=M_{s_0}$ because otherwise we would reach
a contradiction. Hence $M$ is symmetric with restect to $\{t=s_0\}$.
\end{proof}

\end{document}